\documentclass[10pt,righttag]{amsart}
\usepackage{amssymb,verbatim,amscd,mathrsfs,stmaryrd,wasysym,latexsym, bbm}
\usepackage[all]{xy} 
\usepackage{cancel}
\usepackage{lscape}
\usepackage{color}
\begin{document}

\renewcommand{\subjclassname}{%
  \textup{2010} Mathematics Subject Classification}

\newcommand{\V}{{\mathcal V}}      
\renewcommand{\O}{{\mathcal O}}
\newcommand{\LL}{\mathcal L}
\newcommand{\Ext}{\hbox{\rm Ext}}
\newcommand{\Tor}{\hbox{\rm Tor}}
\newcommand{\Hom}{\hbox{Hom}}
\newcommand{\Proj}{\hbox{Proj}}
\newcommand{\GrMod}{\hbox{GrMod}}
\newcommand{\grmod}{\hbox{gr-mod}}
\newcommand{\tors}{\hbox{tors}}
\newcommand{\rank}{\hbox{rank}}
\newcommand{\End}{\hbox{{\rm End}}}
\newcommand{\Der}{\hbox{Der}}
\newcommand{\GKdim}{\hbox{GKdim}}
\newcommand{\im}{\hbox{im}}
\renewcommand{\ker}{\hbox{ker }}
\newcommand{\ev}{\hbox{ev}}
\renewcommand{\sp}{\hbox{span}}
\renewcommand{\char}{\hbox{char }}

\newcommand{\lonto}{{\protect \longrightarrow\!\!\!\!\!\!\!\!\longrightarrow}}

\renewcommand{\c}{\cancel}
\newcommand{\g}{\frak g}
\newcommand{\h}{\frak h}
\newcommand{\fp}{\frak p}
\newcommand{\m}{{\mu}}
\newcommand{\gl}{{\frak g}{\frak l}}
\newcommand{\ssl}{{\frak s}{\frak l}}

\newcommand{\ds}{\displaystyle}
\newcommand{\s}{\sigma}
\renewcommand{\l}{\lambda}
\renewcommand{\a}{\alpha}
\renewcommand{\b}{\beta}
\newcommand{\G}{\Gamma}
\renewcommand{\gg}{\gamma}
\newcommand{\z}{\zeta}
\newcommand{\e}{\epsilon}
\renewcommand{\d}{\delta}
\newcommand{\p}{\rho}
\renewcommand{\t}{\tau}

\newcommand{\C}{{\Bbb C}}
\newcommand{\N}{{\Bbb N}}
\newcommand{\Z}{{\Bbb Z}}
\newcommand{\ZZ}{{\Bbb Z}}
\newcommand{\Q}{{\Bbb Q}}
\renewcommand{\k}{\mathbb K}

\newcommand{\K}{{\mathcal K}}

\newcommand{\rowxy}{(x\ y)}
\newcommand{\colxy}{ \left({\begin{array}{c} x \\ y \end{array}}\right)}
\newcommand{\scolxy}{\left({\begin{smallmatrix} x \\ y
\end{smallmatrix}}\right)}

\renewcommand{\P}{{\Bbb P}}

\newcommand{\la}{\langle}
\newcommand{\ra}{\rangle}
\newcommand{\tensor}{\otimes}
\newcommand{\tsr}{\tensor}

\newtheorem{thm}{Theorem}[section]
\newtheorem{lemma}[thm]{Lemma}
\newtheorem{cor}[thm]{Corollary}
\newtheorem{prop}[thm]{Proposition}

\theoremstyle{definition}
\newtheorem{defn}[thm]{Definition}
\newtheorem{notn}[thm]{Notation}
\newtheorem{ex}[thm]{Example}
\newtheorem{rmk}[thm]{Remark}
\newtheorem{rmks}[thm]{Remarks}
\newtheorem{note}[thm]{Note}
\newtheorem{example}[thm]{Example}
\newtheorem{problem}[thm]{Problem}
\newtheorem{ques}[thm]{Question}
\newtheorem{thingy}[thm]{}

\newcommand{\onto}{{\protect \rightarrow\!\!\!\!\!\rightarrow}}
\newcommand{\donto}{\put(0,-2){$|$}\put(-1.3,-12){$\downarrow$}{\put(-1.3,-14.5) 

{$\downarrow$}}}

\newcounter{letter}
\renewcommand{\theletter}{\rom{(}\alph{letter}\rom{)}}

\newenvironment{lcase}{\begin{list}{~~~~\theletter} {\usecounter{letter}
\setlength{\labelwidth4ex}{\leftmargin6ex}}}{\end{list}}

\newcounter{rnum}
\renewcommand{\thernum}{\rom{(}\roman{rnum}\rom{)}}

\newenvironment{lnum}{\begin{list}{~~~~\thernum}{\usecounter{rnum}
\setlength{\labelwidth4ex}{\leftmargin6ex}}}{\end{list}}


\thispagestyle{empty}

\title{The Koszul property for graded twisted tensor products}


\author[  Conner, Goetz ]{ }

\maketitle

\begin{center}

\vskip-.2in Andrew Conner \\
\bigskip

Department of Mathematics and Computer Science\\
Saint Mary's College of California\\
Moraga, CA 94575\\
\bigskip

 Peter Goetz \\
\bigskip

Department of Mathematics\\ Humboldt State University\\
Arcata, California  95521
\\ \ \\

\end{center}

\setcounter{page}{1}

\thispagestyle{empty}

\vspace{0.2in}

\begin{abstract}

Let $\k$ be a field. Let $A$ and $B$ be connected $\N$-graded $\k$-algebras. Let $C$ denote a twisted tensor product of $A$ and $B$ in the category of connected $\N$-graded $\k$-algebras. The purpose of this paper is to understand when $C$ possesses the Koszul property, and related questions. We prove that if $A$ and $B$ are quadratic, then $C$ is quadratic if and only if the associated graded twisting map has a property we call the unique extension property. We show that $A$ and $B$ being Koszul does not imply $C$ is Koszul (or even quadratic), and we establish sufficient conditions under which $C$ is Koszul whenever both $A$ and $B$ are. We analyze the unique extension property and the Koszul property in detail in the case where $A=\k[x]$ and $B=\k[y]$.

\end{abstract}

\bigskip




\section{Introduction}

Though the study of factorization structures in mathematics has a long history, there has been much recent interest in very general questions about the dual notions of product and factorization for associative algebras over a field.  C\v{a}p, Schichl, and Van\v{z}ura introduced in \cite{Cap} a very general notion of product for a pair of associative algebras, or equivalently, a notion of factorization, called the twisted tensor product.
%
%
This product is analogous to the Zappa-Sz\'{e}p product for groups, see \cite{Brin} for example, and the bicrossed product for Hopf algebras, \cite{Agore2014}. Commutative tensor products, Ore extensions, and smash products of algebras can all be realized as particular cases of twisted tensor products. 
Homological and ring-theoretic properties of these particular cases are well-studied, and a number of recent papers establish such properties of twisted tensor products in the graded setting; see \cite{JPS}, \cite{ShepWi}, \cite{WaWi} for example.

One particularly important homological property of a graded algebra is the Koszul property, introduced by Priddy in \cite{Priddy}. Let $\k$ be a field. Let $A$ be an $\N$-graded $\k$-algebra. We further assume $A$ is connected ($\dim_{\k} A_0=1$) and locally finite dimensional  ($\dim_{\k} A_i<\infty$ for all $i\ge 0$). 
For a $\k$-vector space $V$, let $T(V)$ denote the tensor algebra generated by $V$.

The graded algebra $A$ is called \emph{one-generated} if the canonical multiplication map $\pi:T(A_1)\to A$ is surjective. Let $I=\la (\ker \pi)\cap A_1\tsr A_1\ra$ be the ideal of $T(A_1)$ generated by the degree 2 elements of the kernel of $\pi$. If $A$ is one-generated, the \emph{quadratic part of $A$} is the algebra $q(A)=T(A_1)/I$. If $A\cong q(A)$ then $A$ is called \emph{quadratic}. 

The graded algebra $A$ is called \emph{Koszul} if the trivial module $\k=A_0=A/A_+$  admits a resolution 
$$\cdots\to P_3\to P_2\to P_1\to P_0 \to \k\to 0$$
such that each $P_i$ is a graded free left $A$-module generated in degree $i$. 

Two well-known facts follow immediately from this definition: 
\begin{enumerate}
\item every Koszul algebra is quadratic, 
\item every one-generated, free algebra is Koszul. 
\end{enumerate}

Many important quadratic algebras arising naturally in mathematics are Koszul, and Koszul algebras satisfy a powerful duality property that sometimes underlies deep connections between apparently unrelated problems. One famous example is the duality between the symmetric algebra and the exterior algebra underlying the so-called \emph{BGG correspondence} \cite{BGG}. For additional motivation, background, and examples of Koszul algebras, we encourage the interested reader to consult the book \cite{PP}, which remains the most comprehensive and useful treatment of the theory of quadratic and Koszul algebras we have read. 

Determining whether or not a given quadratic algebra has the Koszul property can be quite difficult. Given the nice homological properties possessed by Koszul algebras, it is therefore natural to study the degree to which the Koszul property is preserved under notions of factorization and product. Our interest in this problem arose from questions about Koszulity of certain smash products of graded Hopf algebras in \cite{CG}. The purpose of the current paper is not to answer questions raised in \cite{CG}, but rather to address some basic questions about the behavior of the Koszul property with regard to twisted tensor products in the category of connected graded algebras. 

A twisted tensor product of $\k$-algebras $A$ and $B$ can be fruitfully studied by considering an associated $\k$-linear \emph{twisting map} $\t:B\tsr A\to A\tsr B$. The twisted tensor product algebra associated to $\t$ is denoted $A\tsr_{\t} B$. In the classical case of a commutative tensor product, $\t(b\tsr a)=a\tsr b$. See Section 2 for more details. If $A$ and $B$ are graded, the $\k$-linear map $\t$ must be a map of graded vector spaces. This raises the question of how the space $A\tsr B$ is graded. If $A$ and $B$ are $\N$-graded, then $A\tsr B$ is $\N\times \N$ graded in the obvious way, but $A\tsr B$ is also $\N$-graded by the K\"{u}nneth grading: $(A\tsr B)_i = \oplus_{p+q=i} A_p\tsr B_q$. The commutative tensor product obviously fits either choice of grading.

The existing literature on graded twisting maps all but exclusively deals with the very restrictive $\N\times \N$ grading, assuming $\t(B_j\tsr A_i)=A_i\tsr B_j$, which further implies $\t$ is invertible. This enables one to prove theorems that closely parallel those for commutative tensor products. However, this restriction also excludes many algebras of interest, including Ore extensions with nontrivial derivations and many smash products of algebras. Indeed, assuming $A$ and $B$ are augmented $\k$-algebras, a twisting map $\t:B\tsr A\to A\tsr B$ induces the structure of a left $B$-module on $A$ via $(1_A \tsr \e_B) \t$ and that of a right $A$-module on $B$ via $(\e_A \tsr 1_B)\t$, where $\e_A$ and $\e_B$ are the augmentation maps on $A$ and $B$, respectively. Requiring $\t$ to preserve the $\N\times \N$ grading forces both of these induced actions to be trivial, leading readily to the following theorem. 

\begin{thm}[(\cite{JPS}, Theorem 4.18), (\cite{WaWi}, Proposition 1.8)]
\label{pure graded Koszul}
Suppose that $A$ and $B$ are Koszul algebras. Let $\t: B \tsr A \to A \tsr B$ be an invertible graded twisting map such that $\t(B_i \tsr A_j) = A_j \tsr B_i$ for all $i, j$. Then the twisted tensor product algebra $A \tsr_{\t} B$ is Koszul.
\end{thm}

It is well known that graded Ore extensions of Koszul algebras are Koszul, indicating this theorem holds more generally. On the other hand, there are simple examples showing that a graded twisted tensor product of Koszul algebras need not be Koszul (see Example \ref{non-Koszul example}). Thus, in this paper, we study the most general type of graded twisting map $\t: B \tsr A \to A \tsr B$. We do not insist that $\t$ be invertible, and we only require the K\"{u}nneth grading to be preserved: $\t(B \tsr A)_n \subseteq (A \tsr B)_n$ for all $n \geq 0$. We provide answers to the following questions.

\begin{enumerate}
\item When are graded twisted tensor products of quadratic algebras quadratic?

\item When are graded twisted tensor products of Koszul algebras Koszul?
\end{enumerate}

Here is an outline of the paper. In Section 2 we recall background, make relevant definitions and extensions of the results in \cite{Cap} and \cite{BorowiecMarcinek} to the graded setting, and discuss the use of filtrations for later use in proving the Koszul property.  

Section 3 is concerned with existence and uniqueness of graded twisting maps. To define a graded twisting map, it is often useful to work inductively. In Lemma \ref{TMEP} we characterize when a graded map which is \emph{twisting to degree $n$} may be extended (uniquely) to a graded map which is twisting to degree $n+1$. Theorem \ref{simpler extension condition} gives very useful and, in practice, checkable sufficient conditions for when a graded map defined to degree $n+1$ which is twisting to degree $n$ is also twisting to degree $n+1$. In the last part of Section 3 we construct a new large class of graded twisting maps $\t: B \tsr A \to A \tsr B$ where $A$ and $B$ are free algebras of arbitrary finite rank, then in Theorem \ref{inducedTau} we determine conditions for when a graded twisting map defined on free algebras induces a graded twisting map on algebras with relations.

In Section 4 we determine when a twisted tensor product of quadratic algebras is quadratic. If $A$ and $B$ are $\N$-graded $\k$-algebras, we say a graded twisting map $\t$ has the \emph{unique extension property} if for all graded twisting maps $\t':B\tsr A\to A\tsr B$ and all $n\in\N$ such that $\t'_i=\t_i$ for $i<n$, then $\t'_n=\t_n$. Our main result in this section is Theorem \ref{UEPquadratic}, which has the following corollary.

\begin{thm}[Corollary \ref{quadTTP}]
If $A$ and $B$ are quadratic, then $\t$ has the unique extension property if and only if $A \tsr_{\t} B$ is quadratic. 
\end{thm}

In Section 5 we consider the Koszul property for graded twisted tensor products. We start by proving, Theorem \ref{oneSidedKoszul}, that the twisted tensor product of Koszul algebras is Koszul for {\it one-sided} twisting maps (if $\t$ is one-sided, then one of the component algebras of the twisted tensor product is normal). We also prove, Proposition \ref{freeKoszul}, that for an arbitrary graded twisting map $\t$, that $A \tsr_{\t} B$ is Koszul if $A$ and $B$ are free algebras, and $A \tsr_{\t} B$ is quadratic. Example \ref{non-Koszul example} shows that the graded twisted tensor product of Koszul algebras need not be Koszul. The remainder of Section 5 is concerned with the introduction of a large class of two-sided twisting maps which we call {\it separable}. We determine, in Theorem \ref{separableKoszul}, sufficient conditions to ensure that $A \tsr_{\t} B$ is Koszul when $\t$ is separable.

\begin{thm}
Let $A$ and $B$ be quadratic algebras. Let $\t:B\tsr A\to A\tsr B$ be a separable graded twisting map. Assume $\t$ satisfies
\begin{enumerate}
\item $\pi_{A_0\tsr A_2\tsr B_1}(1\tsr \t)(\t_B\tsr 1)(B_1\tsr I_2)=0$, and 
\item  $\pi_{A_1\tsr B_2\tsr B_0}(\t\tsr 1)(1\tsr \t_A)(J_2\tsr A_1)=0.$
\end{enumerate}
Let $F$ denote the filtration on $A\tsr_{\t} B$ defined prior to Lemma \ref{extendingR}, and let $F^B$ denote its restriction to the subalgebra $B$. Let $\widetilde{B} = {\rm gr}^{F^B}(B)$. If $A$ is Koszul, the quadratic part of $\widetilde{B}$  is Koszul and $\widetilde{B}$ has no defining relations in degree 3, then $A\tsr_{\t} B$ is Koszul.
\end{thm}

This theorem has the following immediate corollary.

\begin{cor}

Let $A$ be a Koszul algebra with quadratic relation space $I_2$, let $B$ be a free algebra, and suppose $\t: B \tsr A \to A \tsr B$ is a separable graded twisting map. Assume that $\pi_{A_0\tsr A_2\tsr B_1}(1\tsr \t)(\t_B\tsr 1)(B_1\tsr I_2)=0$. Then $A \tsr_{\t} B$ is Koszul.
\end{cor}

In Section 6 we analyze twisting maps $\t: B \tsr A \to A \tsr B$ where $A = \k[x]$ and $B = \k[y]$. Our work here is complemented by the work of Guccione, Guccione and Valqui in \cite{GGV}. We prove the following theorem.
\begin{thm} Suppose that $\t: B \tsr A \to A \tsr B$ is a graded twisting map and write $\t(y \tsr x) = ax^2 \tsr 1 + bx \tsr y + 1 \tsr cy^2$. If $1-ac \ne 0$ and $b \ne 0$ or $c \ne 0$, then $\t$ has the unique extension property and $A \tsr_{\t} B$ is quadratic, hence Koszul.
\end{thm}
We also consider when the definition $\t(y \tsr x) = x^2 \tsr 1 + 1 \tsr cy^2$ extends to a graded twisting map, culminating in a complete answer in Theorem 6.5 and Proposition 6.6. Our characterization exhibits a curious connection to the Catalan numbers.

Finally in Section 7 we consider three examples that illustrate the main theorems of the paper. Most importantly Example 7.4 is an example of a twisted tensor product that arises from a separable twisting map that is not isomorphic to any twisted tensor product coming from a one-sided twisting map. 

\section{Preliminaries}


As discussed above, this paper explores the transfer of the Koszul property between a twisted tensor product and its component subalgebras, which need not be normal. In this section we recall the relevant background.

\subsection{Twisted products and twisting maps}
%
%

Throughout the paper, let $\k$ denote a field. Tensor products taken with respect to $\k$ are denoted by $\tsr$. If $V$ is a $\k$-vector space, we write $1_V$ for the identity map on $V$; if $V$ is a unital $\k$-algebra we will abuse notation and also write $1_V \in V$ for the identity element.

Let $A$ and $B$ be (unital) $\k$-algebras with multiplication maps $\mu_A$ and $\mu_B$. Following \cite{Cap},  an \emph{internal twisted tensor product} of $A$ and $B$ is a triple $(C,i_A,i_B)$ where $C$ is a $\k$-algebra and  $i_A:A\to C$ and $i_B:B\to C$ are injective $\k$-algebra homomorphisms such that the $\k$-linear map $A\tsr B\to C$ given by $a\tsr b \mapsto i_A(a)i_B(b)$ is an isomorphism of $\k$-vector spaces. We say two internal twisted tensor products $(C,i_A,i_B)$ and $(C',i_A', i_B')$ of (graded) $\k$-algebras $A$ and $B$ are \emph{isomorphic} if there exist (graded) algebra homomorphisms $\alpha:A\to A$, $\beta:B\to B$, and $\gamma:C\to C'$ such that $\gamma i_A=i_A'\alpha$, $\gamma i_B=i_B'\beta$, and $\gamma$ is an isomorphism. 


We call a $\k$-linear map $\t:B\tsr A\to A\tsr B$ an \emph{algebra twisting map} if 
$\t(1_B\tsr a)=a\tsr 1_B$ and $\t(b\tsr 1_A)=1_A\tsr b$ for all $a\in A$ and $b\in B$ and if
$$\t(\mu_B\tsr \mu_A)=(\mu_A\tsr \mu_B)(1_A\tsr \t\tsr 1_B)(\t\tsr \t)(1_B\tsr \t\tsr 1_A).$$ This definition was introduced in \cite{Cap}. We will refer to the conditions $\t(1_B\tsr a)=a\tsr 1_B$ and $\t(b\tsr 1_A)=1_A\tsr b$ for all $a\in A$ and $b\in B$ as the \emph{unital twisting conditions}. For any $\k$-linear map $\t:B\tsr A\to A\tsr B$ we adopt Sweedler-type notation and write $\t(b\tsr a) = a_{\t} \tsr b_{\t}$.

If $A$ and $B$ carry a grading by the semigroup $\N$, the $\k$-linear tensor product $A\tsr B$ admits an $\N$-grading by the K\"{u}nneth formula
$$(A\tsr B)_{m} = \bigoplus_{k+l=m} A_{k}\tsr B_{l}.$$

More generally, if $V$ and $W$ are $\N$-graded $\k$-vector spaces, we grade $V \tsr W$ by the K\"{u}nneth formula.
Throughout this paper the term \emph{graded} will mean \emph{$\N$-graded}. A $\k$-linear map $f:V\to W$ is called \emph{graded} if it preserves the $\N$ grading: $f(V_i)\subseteq W_i$ for all $i\ge 0$. If an algebra twisting map $\t:B\tsr A\to A\tsr B$ is graded we call it a \emph{graded algebra twisting map}, or just \emph{graded twisting map}. If $C$ is a graded algebra, we write $C_+$ for $\displaystyle{\oplus_{i > 0} C_i}$.


If $V$ and $W$ are graded $\k$-vector spaces and $f: V \to W$ is a graded $\k$-linear map, we denote the degree $n$ component of $f$ by $f_n$ and define $f_{\leq n} = \bigoplus_{i=0}^n f_i$ and $f_{>n} = \bigoplus_{i>n} f_i$.


We say a graded $\k$-linear map $t:(B\tsr A)_{\le n} \to (A\tsr B)_{\le n}$ is \emph{graded twisting in degree $n$} if $t(1_B\tsr a)=(a\tsr 1_B)$ and $t(b\tsr 1_A)=(1_A\tsr b)$ for all $a\in A_n$ and $b\in B_n$ and
$$t_{n} (\mu_{B}\tsr \mu_{A})=(\mu_A\tsr \mu_B) (1_A\tsr t_{\le n}\tsr 1_B) (t_{\le n}\tsr t_{\le n})(1_B\tsr t_{\le n}\tsr 1_A)$$ as maps defined on $(B \tsr B \tsr A \tsr A)_n$. 

If $t$ is graded twisting in degree $i$ for all $i\le n$ we say $t$ is \emph{graded twisting to degree $n$}. Evidently, a $\k$-linear map $\t:B\tsr A\to A\tsr B$ is a graded twisting map if and only if the restriction of $\t$ to $(B\tsr A)_{\le n}$ is graded twisting to degree $n$ for all $n\ge 0$ (which is equivalent to graded twisting in degree $n$ for all $n\ge 0$). Throughout the paper, we adopt the convention of using $\t$ to describe potential graded twisting maps, that is, graded $\k$-linear maps of the form $B\tsr A\to A\tsr B$. We use $t$ to indicate a graded $\k$-linear map defined only on $(B\tsr A)_{\le n}\to (A\tsr B)_{\le n}$, which is potentially graded twisting to degree $n$.

%

\begin{rmk}
\label{AltTwistingDef}
Let $A$ and $B$ be $\k$-algebras and $\t:B\tsr A\to A\tsr B$ a $\k$-linear map. As noted in \cite{Cap} (Remark 2.4(1)) $\t$ is a twisting map if and only if the following identities hold.
\begin{align*}
\t(1_B\tsr \mu_A) &= (\mu_A\tsr 1_B)(1_A\tsr\t)(\t\tsr 1_A)\\
\t(\mu_B\tsr 1_A) &= (1_A\tsr \mu_B)(\t\tsr 1_B)(1_B\tsr \t)
\end{align*}
The reasoning in \cite{Cap} applied to homogeneous components of $\t$ in the graded setting shows that $t:(B\tsr A)_{\le n}\to (A\tsr B)_{\le n}$ is graded twisting to degree $n$ if and only if $t(1_B\tsr a)=(a\tsr 1_B)$ and $t(b\tsr 1_A)=(1_A\tsr b)$ for all $a\in A_{\leq n}$ and $b\in B_{\leq n}$, and for all $j \leq n$,
\begin{align*}
t_{j}(1_B\tsr \mu_A) &= (\mu_A\tsr 1_B)(1_A\tsr t_{\le j})(t_{\le j}\tsr 1_A)\\
t_{j}(\mu_B\tsr 1_A) &= (1_A\tsr \mu_B)(t_{\le j}\tsr 1_B)(1_B\tsr t_{\le j}).
\end{align*}
\end{rmk}

For later use we record the following proposition. The easy proof is left to the reader.

\begin{prop}
\label{isomorphism type of ttp}
Let $A$ and $B$ be graded $\k$-algebras, and let $\a: A \to A$ and $\b: B \to B$ be graded automorphisms. 

Suppose that $t: (B \tsr A)_{\leq n} \to (A \tsr B)_{\leq n}$ is a $\k$-linear map that is graded twisting to degree $n$. Define $t': (B \tsr A)_{\leq n} \to (A \tsr B)_{\leq n}$ by $t' = (\a \tsr \b) t(\b^{-1} \tsr \a^{-1})|_{(B \tsr A)_{\leq n}}$. Then $t'$ is graded twisting to degree $n$. Moreover, there exists a unique $\k$-linear extension of $t$, $\widehat{t}: (B \tsr A)_{\leq n+1} \to (A \tsr B)_{\leq n+1}$ that is graded twisting to degree $n+1$ if and only if there exists a unique $\k$-linear extension of $t'$, $\widehat{t'}: (B \tsr A)_{\leq n+1} \to (A \tsr B)_{\leq n+1}$ that is graded twisting to degree $n+1$. 

If $\t: B \tsr A \to A \tsr B$ is a graded twisting map, then $\t': B \tsr A \to A \tsr B$ defined by $\t' = (\a \tsr \b) \t(\b^{-1} \tsr \a^{-1})$ is a graded twisting map. Furthermore, the algebras $A \tsr_{\t} B$ and $A \tsr_{\t'} B$ are isomorphic as twisted tensor products of $A$ and $B$.
\end{prop}

The relationship between twisting maps and twisted products was established for $\k$-algebras in \cite{Cap}. Here we add graded and truncated algebra versions of the theorem as well. 

\begin{prop}
\label{multiplications}
Let $A$ and $B$ be (graded) algebras. Let $\t: B \tsr A \to A \tsr B$ be a (graded) $\k$-linear map. Define a map $\m_{\t}: A \tsr B \tsr A \tsr B \to A \tsr B$ by $\mu_{\t}=(\mu_A\tsr \mu_B)(1_A\tsr \t\tsr 1_B)$. 

\begin{enumerate}

\item The map $\t:B\tsr A\to A\tsr B$ is a (graded) algebra twisting map if and only if $\mu_{\t}$ defines an associative multiplication giving $A\tsr B$ the structure of a (graded) algebra.

\item Assume that $A$ and $B$ are graded and $\t:B\tsr A\to A\tsr B$ is graded. Then $\t$ is twisting to degree $n$ if and only if $\mu_{\t}$ induces an associative multiplication giving $(A\tsr B)/(A\tsr B)_{>n}$ the structure of a graded algebra.
\end{enumerate}
\end{prop}

\begin{proof}
Statement (1) in the not necessarily graded case is Proposition 2.3 of \cite{Cap}; for the graded case of (1), one follows the proof of Proposition 2.3 of \cite{Cap} and makes the appropriate changes. 

For (2) we set $C = A \tsr B/(A \tsr B)_{>n}$.  The map $\m_{\t}$ is graded, so, in particular, $$\m_{\t}((A \tsr B)_{>n} \tsr (A \tsr B)_{>n}) \subseteq (A \tsr B)_{>n},$$ hence we get a canonically defined graded $\k$-linear map $\overline{\m_{\t}}: C \tsr C \to C$. We identify $C$ with $(A \tsr B)_{\leq n}$ in the obvious way. Then note that under this identification $\overline{\m_{\t}}_{>n} = 0$ and $$\overline{\m_{\t}}_{\leq n} = (\m_A \tsr \m_B)(1_A \tsr \t_{\leq n} \tsr 1_B).$$ Now, to prove that $\overline{\m_{\t}}$ is associative if and only if $\t$ is twisting to degree $n$, one simply follows the proof of Proposition 2.3 in \cite{Cap}.
\end{proof}

We denote the algebra determined by $\t$ in Proposition \ref{multiplications} (1) by $A\tsr_{\t} B$. We call $A\tsr_{\t} B$ the \emph{external twisted tensor product} of $A$ and $B$.

A given pair of $\k$-algebras may have non-isomorphic external twisted products, and a given $\k$-algebra may be expressed as an internal twisted product in more than one way.

%

\begin{prop}
Let $(C,i_A,i_B)$ be a (graded) twisted tensor product of (graded) $\k$-algebras $A$ and $B$. Then there is a unique (graded) twisting map $\t$ such that $C$ is isomorphic to $A\tsr_{\t} B$ as a (graded) twisted tensor product.
\end{prop}

\begin{proof}
The ungraded version of this statement is Proposition 2.7 in \cite{Cap}. The proof in the graded case is entirely analogous.
\end{proof}

If $A$ and $B$ are graded, a graded twisting map $\t$ satisfies
$$\t(B_+\tsr A_+)\subseteq (A_+\tsr B_0) \oplus (A_+\tsr B_+) \oplus (A_0\tsr B_+).$$
 We will see in Section \ref{Koszul twisted tensor products} that a few special cases are important to distinguish. If
$$\t(B_+\tsr A_+)\subseteq (A_+\tsr B_0) \oplus (A_+\tsr B_+)$$ or $$\t(B_+\tsr A_+)\subseteq  (A_+\tsr B_+) \oplus (A_0\tsr B_+)$$
we call the graded twisting map $\t$ \emph{one-sided}; if $\t(B_+\tsr A_+)\subseteq (A_+\tsr B_+)$ we call $\t$ \emph{pure}; if $\t(B_i \tsr A_j) \subseteq A_j \tsr B_i$ for all $i, j \geq 0$, then we call $\t$ \emph{strongly graded}. In our experience, graded twisting maps can behave rather badly in general. For example, from a homological point of view, in general, one cannot naturally form a projective resolution of $A \tsr_{\t} B$ from projective resolutions of $A$ and $B$, unless one uses bar complexes, see \cite{GucGuc}. The class of one-sided twisting maps is much nicer; and pure and strongly graded twisting maps are nicer still.  It is the last case which is most often encountered in the literature; \cite{JPS}, \cite{ShepWi}, \cite{WaWi}.
%
%
We implore the reader to bear in mind: 
\bigskip

\noindent \emph{Many references to graded twisting maps in the literature in fact refer to the more restrictive cases of one-sided, pure, or strongly graded twisting maps.} 
\bigskip

In light of this potential for confusion, when we wish to emphasize that a statement about graded twisting maps applies in full generality, we refer to the twisting map as \emph{two-sided}.

In \cite{BorowiecMarcinek}, the authors identify the twisted tensor product $A\tsr_{\t} B$ with a certain quotient of the free product algebra, $A\ast B$.  As a vector space $$A \ast B = \bigoplus_{i \geq 0; \e_1, \e_2 \in \{0 ,1\}} A_+^{\e_1} \tsr (B_+ \tsr A_+)^{\tsr i} \tsr B_+^{\e_2}.$$ 
We note that $A \ast B$ is $\N$-graded with the usual K\"unneth grading. 

We will need a slight generalization of the graded version of their result. We refer to an element $b\tsr a-\t(b\tsr a)\in A\ast B$ as a \emph{$\t$-relation}.

For a graded $\k$-linear map $\t:B\tsr A\to A\tsr B$ and $n\in\N$, we let $I_{\t}^n$ be the ideal of $A\ast B$ generated by $\t$-relations of degree at most $n$. That is,
$$I_{\t}^n = \la b\tsr a - \t(a \tsr b) \ |\ a\in A_i, b\in B_j, i+j\le n\ra.$$
The ideal generated by all $\t$-relations is
$$I_{\t} = \la b\tsr a - \t(a \tsr b)\ |\ a\in A, b\in B\ra.$$

\begin{prop}
\label{presentation}
Let $A$ and $B$ be graded, unital $\k$-algebras. If the $\k$-linear map $\t:B\tsr A\to A\tsr B$ is graded twisting to degree $n$, then there is a $\k$-algebra homomorphism $\varphi:(A\ast B)/I_{\t}^n\to (A\tsr B)/(A\tsr B)_{>n}$, where the algebra structure on $(A\tsr B)/(A\tsr B)_{>n}$ is given by Proposition \ref{multiplications}(2).  The map $\varphi$ is an isomorphism in degrees $\le n$.   

If $\t$ is a graded twisting map, then $A\tsr_{\t} B\cong (A\ast B)/I_{\t}$.
\end{prop}

\begin{proof}
First assume that $\t$ is a graded twisting map. Using the universal property of the free product there exists a surjective graded algebra homomorphism $$\pi: A \ast B \to A \tsr_{\t} B.$$ It is clear that the ideal $I_{\t}$ is contained in $\ker \pi$. Moreover, it is evident from the defining generators of $I_{\t}$ that in each graded component $$\dim((A \ast B)/I_{\t})_i \leq \dim(A \tsr_{\t} B)_i.$$ Therefore we have $\ker \pi = I_{\t}$ and $A\tsr_{\t} B\cong (A\ast B)/I_{\t}$.

Now suppose that $\t:B\tsr A\to A\tsr B$ is graded twisting to degree $n$. From Proposition \ref{multiplications} (2), $(A \tsr B)/(A \tsr B)_{>n}$ is canonically an associative algebra. There are obvious algebra maps $A \to (A \tsr B)/(A \tsr B)_{>n}$ and $B \to (A \tsr B)/(A \tsr B)_{>n}$ so the universal property of $A \ast B$ affords a surjective algebra homomorphism $A \ast B \to (A \tsr B)/(A \tsr B)_{>n}$ whose kernel contains $I_{\t}^n$. Comparing dimensions of graded components in degrees $\leq n$ we see there is an induced map $$\varphi:(A\ast B)/I_{\t}^n\to (A\tsr B)/(A\tsr B)_{>n}$$ which is an isomorphism in degrees $\leq n$.
\end{proof}

\subsection{Quadratic and Koszul algebras} 
\label{KoszulPrelims}

Recall that we assume $A$ is an $\N$-graded $\k$-algebra that is connected and locally finite dimensional. As mentioned above, the graded algebra $A$ is called \emph{Koszul} if the trivial module $\k=A_0=A/A_+$  admits a resolution 
$$\cdots\to P_3\to P_2\to P_1\to P_0 \to \k\to 0$$
such that each $P_i$ is a graded free left $A$-module generated in degree $i$.

There are several formulations of the notion of Koszul algebra that are equivalent to the definition above, and these lead to a multitude of techniques for proving an algebra is Koszul. One very useful approach makes use of filtrations.

Let $\G$ be a graded ordered semigroup. By definition, $\G$ is a semigroup with unit $e$ endowed with a semigroup map $g: \G \to \N$ such that the fibers $\G_n = g^{-1}(n)$ are totally ordered with their orders satisfying:

$$\text{ for all } \a, \b \in \G_k \text{ and } \gamma \in \G_l, \ \a < \b \implies \a \gamma < \b \gamma \text{ and } \gamma \a < \gamma \b.$$ We also assume $g^{-1}(0) = \{e\}$.

 A $\G$-valued filtration on a graded algebra $A$ is a collection of subspaces $F_{\a}A_n\subset A_n$ for all $n\ge 0$ and $\a\in \G_n$ such that 
\begin{itemize}
\item $F_{\a}A_n\subseteq F_{\b} A_m$ whenever $\a\le \b$,
\item if $\gamma \in \G_n$ is maximal, then $F_{\gamma}A_n=A_n$,
\item $F_{\a}A_n \cdot F_{\b} A_m \subseteq F_{\a\b} A_{n+m}$.
\end{itemize}
The associated $\G$-graded algebra is ${\rm gr}^F A =\bigoplus_{\a\in\G} (F_{\a} A_n/(\sum_{\a' < \a} F_{\a'}A_n))$. The filtration $F$ is called \emph{one-generated} if 
$$F_{\a}A_n = \sum_{i_1i_2\cdots i_k\le\a} F_{i_1}A_1\cdot F_{i_2}A_1\cdots F_{i_k}A_1$$
Observe that ${\rm gr}^F A$ is one-generated if and only if $F$ is. The next theorem illustrates the usefulness of filtrations for studying Koszul algebras.

\begin{thm}[\cite{PP}, Theorem 7.1, p.\ 89] 
\label{filtrationTrick}
Let $A$ be a quadratic algebra equipped with a one-generated filtration $F$ with values in a graded ordered semigroup $\G$. Assume the associated $\G$-graded algebra ${\rm gr}^F A$ satisfies the following conditions
\begin{enumerate}
\item ${\rm gr}^F A$ has no defining relations of degree 3;
\item $q({\rm gr}^F A)$ is Koszul
\end{enumerate}
Then ${\rm gr}^F A$ is quadratic (hence Koszul) and $A$ is Koszul.
\end{thm}

The theorem is particularly useful in the case of a normal subalgebra. 

\begin{cor}[\cite{PP}, Example 2, p.\ 90]\label{one-sided Koszul}
Let $f:A\to C$ be a homomorphism of quadratic algebras such that $C_1f(A_1)\subset f(A_1)C_1$. Assume that $A$ and $C/f(A_1)C$ are Koszul algebras and the left action of $A$ on $C$ is free in degrees $\le 3$. Then $C$ is a Koszul algebra.
\end{cor}

\section{Existence of graded twisting maps}
\label{existence}

In the graded setting, it is often convenient to define graded twisting maps inductively. We now describe the inductive step in this process, which we refer to as the \emph{twisting map extension problem}. Throughout this section let $A$ and $B$ denote connected, graded $\k$-algebras.

Let $t:(B\tsr A)_{\le n} \to (A\tsr B)_{\le n}$ be a $\k$-linear map that is graded twisting to degree $n$. 
Let $t^A = \pi^A\circ t$, where $\pi^A$ is the composition of canonical maps 
$$(A \tsr B)_{\le n} \to \bigoplus_{i\le n} A_i\tsr B_0\to A_{\le n}.$$ 
Analogously, let $\pi^B$ be the composition
$$(A \tsr B)_{\le n} \to \bigoplus_{i\le n} A_0\tsr B_i\to B_{\le n}$$ and put $t^B = \pi^B\circ t$.
Let $i_A:A\to A\tsr B$ and $i_B:B\to A\tsr B$ be the canonical inclusions.

Let $\d:(B_+\tsr A_+\tsr A_+)_{\le n+1}\oplus (B_+\tsr B_+\tsr A_+)_{\le n+1} \to (B_+\tsr A_+)_{\le n+1}$ be the $\k$-linear map defined by
$$\d=(1_B\tsr \mu_A - t^B\tsr 1_A)\oplus (\mu_B\tsr 1_A - 1_B\tsr t^A)$$
and let $R:(B_+\tsr A_+\tsr A_+)_{\le n+1}\oplus (B_+\tsr B_+\tsr A_+)_{\le n+1} \to (A\tsr B)_{\le n+1}$ be given by
\begin{align*}
R&=(\mu_A\tsr 1_B)(1_A\tsr t)((t-i_Bt^B)\tsr 1_A)\\
&\oplus (1_A\tsr \mu_B)(t\tsr 1_B)(1_B\tsr(t-i_At^A))
\end{align*}
It is important to note that in this context the compositions $(1_A\tsr t)((t-i_Bt^B)\tsr 1_A)$ and $(t\tsr 1_B)(1_B\tsr(t-i_At^A))$ are well-defined, but if $t$ is two-sided, $(1_A\tsr t)(t\tsr 1_A)$ and $(t\tsr 1_B)(1_B\tsr t)$ may not be.

\begin{lemma}[Twisting Map Extension Problem]
\label{TMEP}
Let $t:(B\tsr A)_{\le n} \to (A\tsr B)_{\le n}$ be a $\k$-linear map that is graded twisting to degree $n$.
There exists a $\k$-linear extension $t':(B\tsr A)_{\le n+1}\to (A\tsr B)_{\le n+1}$ of $t$ that is graded twisting to degree $n+1$ if and only if there exists a $\k$-linear map $f:(B_+\tsr A_+)_{n+1}\to (A\tsr B)_{n+1}$ satisfying $f\d_{n+1}=R_{n+1}$. If such an extension $t'$ exists, it is unique if and only if $\d_{n+1}$ is surjective.
\end{lemma}

If $\t$ is a graded twisting map, then the twisting map extension problem has a solution for every $n$. We will show shortly that non-uniqueness of these extensions is captured precisely by the minimal generators of $I_{\t}$. 


\begin{proof}
First, suppose $t'$ is any $\k$-linear extension of $t$. Note that since $\im\ i_Bt^B\subseteq A_0\tsr B$, 
$$(\mu_A\tsr 1_B)(1_A\tsr t')(i_Bt^B\tsr 1_A)= t'(t^B\tsr 1_A) \ \ \ (1).$$ 

Now assume that $t'$ is graded twisting to degree $n+1$ and put $f=t'|_{(B_+\tsr A_+)_{n+1}}$.
By Remark \ref{AltTwistingDef},
\begin{align*}
f(1_B\tsr \mu_A-t^B\tsr 1_A)_{n+1} &= (\mu_A\tsr 1_B)(1_A\tsr t')(t'\tsr 1_A)_{n+1}- t'(t^B\tsr 1_A)_{n+1}\\
&= (\mu_A\tsr 1_B)(1_A\tsr t')((t'-i_Bt^B)\tsr 1_A)_{n+1}.
\end{align*}
On elements of $(B_+\tsr A_+\tsr A_+)_{n+1}$, we have $t'\tsr 1_A=t\tsr 1_A$. Furthermore, 
$$((t-i_Bt^B)\tsr 1_A)(B_+\tsr A_+\tsr A_+)_{n+1} \subseteq (A_+\tsr B\tsr A_+)_{n+1}.$$
Thus,
\begin{align*} 
f(1_B \tsr \m_A-t^B\tsr 1_A)_{n+1} &= (\mu_A\tsr 1_B)(1_A\tsr t')((t'-i_Bt^B)\tsr 1_A)_{n+1} \\
&=(\mu_A\tsr 1_B)(1_A\tsr t)((t-i_Bt^B)\tsr 1_A)_{n+1}
\end{align*}
on $(B_+\tsr A_+\tsr A_+)_{n+1}$. An analogous calculation shows
$$f(\mu_B\tsr 1 - 1\tsr t^A)_{n+1}=(1_A\tsr \mu_B)(t\tsr 1_B)(1_B\tsr(t-i_At^A))_{n+1}$$
on $(B_+\tsr B_+\tsr A_+)_{n+1}$. Hence $f\d_{n+1}=R_{n+1}$.

Now suppose there exists a $\k$-linear map $f:(B_+\tsr A_+)_{n+1}\to (A\tsr B)_{n+1}$ such that $f\d_{n+1}=R_{n+1}$. Extend $f$ to $f':(B\tsr A)_{n+1} \to (A\tsr B)_{n+1}$ by $f'(1_B\tsr a) = a\tsr 1_B$ and $f'(b\tsr 1_A)=1_A\tsr b$ for all $a\in A_{n+1}$ and $b\in B_{n+1}$. Define $t'=t\oplus f'$. We must show $t'$ is graded twisting to degree $n+1$. It suffices to show $t'$ is graded twisting in degree $n+1$. We have,

\begin{align*}
t'(1_B\tsr \mu_A)_{n+1} &= f'(1_B\tsr \mu_A)_{n+1}\\ 
&= t'(t^B\tsr 1_A)+(\mu_A\tsr 1_B)(1_A\tsr t)((t-i_Bt^B)\tsr 1_A)_{n+1}\\
&= (\mu_A\tsr 1_B)(1_A\tsr t')(i_Bt^B\tsr 1_A)+(\mu_A\tsr 1_B)(1_A\tsr t')((t'-i_Bt^B)\tsr 1_A)_{n+1}\\
&=(\mu_A\tsr 1_B)(1_A\tsr t')(t'\tsr 1_A)_{n+1}.
\end{align*}
The second equality follows from $f\d_{n+1}=R_{n+1}$ and the third is given by (1) above.
%
An analogous calculation shows 
$$t'(\mu_B\tsr 1_A)_{n+1}=(1_A\tsr \mu_B)(t'\tsr 1_B)(1_B\tsr t')_{n+1}.$$
By Remark \ref{AltTwistingDef}, this completes the proof of the existence part of the Lemma.

If $\d_{n+1}$ is surjective, let $f'$ be another linear map satisfying $f'\d_{n+1}=R_{n+1}$. Then for any $x\in (B\tsr A)_{n+1}$, we have $(f-f')(x) = (f-f')\d(y) = (R-R)(y)=0$. So $f$ is unique.

If $\d_{n+1}$ is not surjective, let $W$ be a (nontrivial) complementary subspace of $\im\ \d_{n+1}$ in $(B\tsr A)_{n+1}$. Define $f'=f$ on $\im\ \d$ and freely define $f'$ on $W$ such that $f'|_W\neq f|_W$. Then $f'\d_{n+1}=R$ but $f'\neq f$.
\end{proof}

The following simpler condition for extending a linear map that is graded twisting to some degree is often useful in practice.

\begin{thm}
\label{simpler extension condition}
Let $A$ and $B$ be one-generated $\k$-algebras. Let $n$ be a positive integer. Suppose that $t: (B \tsr A)_{\leq n+1} \to (A \tsr B)_{\leq n+1}$ is a graded $\k$-linear map such that $t$ satisfies the unital conditions: $t(1 \tsr a) = a \tsr 1$ and $t(b \tsr 1) = 1 \tsr b$, and 
\begin{enumerate}
\item[(1)] $t \big|_{(B \tsr A)_{\leq n}}$ is a twisting map to degree $n$;
\item[(2)] $t(1_B \tsr \m_A) = (\m_A \tsr 1_B)(1_A \tsr t)(t \tsr 1_A)$ on $(B \tsr A \tsr A_1)_{n+1}$;
\item[(3)] $t(\m_B \tsr 1_A) = (1_A \tsr \m_B)(t \tsr 1_B)(1_B \tsr t)$ on $(B_1 \tsr B \tsr A)_{n+1}$.
\end{enumerate}
Then $t$ is a twisting map to degree $n+1$.
\end{thm}

\begin{proof}

Let $f = t |_{(B_+ \tsr A_+)_{n+1}}$. We must show $f \d_{n+1} = R_{n+1}$. We first prove that     $f(1_B \tsr \m_A) = R_{n+1} + f(t^B \tsr 1_A).$

Let $b \tsr a_1 \tsr a_2 \in (B_+ \tsr A_+ \tsr A_+)_{n+1}$, where $\deg(a_2) \geq 2$. Write $t(b \tsr a_1) = l(b \tsr a_1) + \sum_i 1_A \tsr b_i,$ where $l(b \tsr a_1) = (t-i_B t^B)(b \tsr a_1) \in A_+ \tsr B$ and $\sum_i 1_A \tsr b_i = i_B t^B(b \tsr a_1)$.

%
%
%
%
%

Since $A$ is one-generated, we have $a_2 = \m_A(\sum_j a_j' \tsr x_j)$ for some $x_j \in A_1$. 

We compute
\begin{align*}
t(b \tsr a_1 a_2) &= t(b \tsr a_1\m_A(\sum_j a_j' \tsr x_j)) \\
&= \sum_j t(b \tsr \m_A(a_1a_j' \tsr x_j)) \\
&= \sum_j t(1_B \tsr \m_A)(b \tsr a_1a_j' \tsr x_j) \\
&= \sum_j(\m_A \tsr 1_B)(1_A \tsr t)(t \tsr 1_A)(b \tsr a_1a_j' \tsr x_j) \text{ using condition } (2),\\
&= \sum_j (\m_A \tsr 1_B)(1_A \tsr t)(t(b \tsr a_1a_j') \tsr x_j) \\
&= \sum_j (\m_A \tsr 1_B)(1_A \tsr t)[(\m_A \tsr 1_B)(1_A \tsr t)(t \tsr 1_A)(b \tsr a_1 \tsr a_j') \tsr x_j] \text{ using condition } (1), \\
&= \sum_j (\m_A \tsr 1_B)(1_A \tsr t)[(\m_A \tsr 1_B)(1_A \tsr t)((l(b \tsr a_1) + \sum_i 1 \tsr b_i) \tsr a_j')\tsr x_j] \\
&= \sum_j (\m_A \tsr 1_B)(1_A \tsr t)[(\m_A \tsr 1_B)(1_A \tsr t)((l(b \tsr a_1) \tsr a_j') \tsr x_j] )\\ 
& \ \ + \sum_j(\m_A \tsr 1_B)(1_A \tsr t)[\sum_i(\m_A \tsr 1_B)(1_A \tsr t)(1 \tsr b_i \tsr a_j') \tsr x_j].
\end{align*}

We next make two claims.

$$\sum_j(\m_A \tsr 1_B)(1_A \tsr t)[(\m_A \tsr 1_B)(1_A \tsr t)(l(b \tsr a_1) \tsr a_j') \tsr x_j] = R_{n+1}(b \tsr a_1 \tsr a_2) \text{ (i)},$$

$$\sum_j(\m_A \tsr 1_B)(1_A \tsr t)[\sum_i(\m_A \tsr 1_B)(1_A \tsr t)(1 \tsr b_i \tsr a_j')\tsr x_j] = f(t^B \tsr 1_A)(b \tsr a_1 \tsr a_2) \text{ (ii)}.$$


Before starting the proof of (i) we note that $$(1_A \tsr t)[(\mu_A \tsr 1_B)(1_A \tsr t)\tsr 1_A] = (\m_A \tsr 1_A \tsr 1_B)(1_A \tsr (1_A \tsr t)(t \tsr 1_A)).$$ 


Now we prove (i).

\begin{align*}
&R_{n+1}(b \tsr a_1 \tsr a_2) \\
&= (\m_A \tsr 1_B)(1_A \tsr t)(l(b \tsr a_1) \tsr a_2) \\
&= (\m_A \tsr 1_B)(1_A \tsr t(1_B \tsr \m_A))(l(b \tsr a_1) \tsr \sum_j a_j' \tsr x_j) \\
&= (\m_A \tsr 1_B)(1_A \tsr (\m_A \tsr 1_B)(1_A \tsr t)(t \tsr 1_A))(\sum_j l(b \tsr a_1) \tsr a_j' \tsr x_j) \text{ using condition } (1), \\
&= (\m_A \tsr 1_B)( \m_A \tsr 1_A \tsr 1_B)(1_A \tsr (1_A \tsr t)(t \tsr 1_A))(\sum_j l(b \tsr a_1) \tsr a_j' \tsr x_j) \text{ since } \m_A \text{ is associative, } \\
&= (\m_A \tsr 1_B)(1_A \tsr t)[(\m_A \tsr 1_B)(1_A \tsr t) \tsr 1_A](\sum_j l(b \tsr a_1) \tsr a_j' \tsr x_j) \text{ using the note above, } \\
& = \sum_j(\m_A \tsr 1_B)(1_A \tsr t)[(\m_A \tsr 1_B)(1_A \tsr t) \tsr 1_A](l(b \tsr a_1) \tsr a_j' \tsr x_j), \text{ as desired.}
\end{align*}

The proof of (ii) is more transparent.

\begin{align*} 
&\sum_j(\m_A \tsr 1_B)(1_A \tsr t)[\sum_i(\m_A \tsr 1_B)(1_A \tsr t)(1 \tsr b_i \tsr a_j') \tsr x_j] \\
 & = \sum_j(\m_A \tsr 1_B)(1_A \tsr t)[(\m_A \tsr 1_B)(1_A \tsr t) \tsr 1_A](\sum_i 1 \tsr b_i \tsr a_j' \tsr x_j) \\
&= \sum_j( \m_A \tsr 1_B)(1_A \tsr t)(t \tsr 1_A)(\sum_i b_i \tsr a_j' \tsr x_j) \\
&= \sum_i t(b_i \tsr a_2) \text{ using condition } (2), \\
&= f(t^B \tsr 1_A)(b \tsr a_1 \tsr a_2).
\end{align*}

We leave it to the reader to check that $f(\m_B \tsr 1_A) = R_{n+1}+f(1_B \tsr t^A)$. 
\end{proof}

In some cases one can construct graded twisting maps on algebras with relations by first defining a twisting map on free algebras and then checking that the ideals of relations are preserved. We first address the case of free algebras, in particular, we examine the question of when defining a linear map on free generators extends to a graded twisting map on free algebras. (We consider the uniqueness of such extensions in the next section.)

We start with some background material following \cite{Cap}. If $V$ and $W$ are $\k$-vector spaces, let $L(V, W)$ denote the space of all linear maps from $V$ to $W$. Let $A$ and $B$ be associative, unital, $\k$-algebras and let $V$ be a $\k$-vector space. Then $L(V, V \tsr B)$ and $L(V, A \tsr V)$ are associative, unital, $\k$-algebras. Their unit elements are $v \mapsto v \tsr 1_B$ and $v \mapsto 1_A \tsr v$, respectively. The multiplication in $L(V, V \tsr B)$ is given by $$\phi \star \psi = (1_V \tsr \mu_B)(\phi \tsr 1_B) \psi;$$ the multiplication in $L(V, A \tsr V)$ is given by $$\phi \star \psi = (\mu_A \tsr 1_V)(1_A \tsr \psi)\phi.$$

\begin{prop}[\cite{Cap}, Proposition 2.6] 
Let $\t: B \tsr A \to A \tsr B$ be a $\k$-linear map. Define linear maps $\t^B: B \to L(A, A \tsr B)$ and $\t^A: A \to L(B, A \tsr B)$ by, for all $a \in A$ and $b \in B$, $\t^B(b)(a) = \t(b \tsr a)$ and $\t^A(a)(b) = \t(b \tsr a)$. Then $\t$ is a twisting map if and only if $\t^A$ and $\t^B$ are algebra homomorphisms.
\end{prop}

Now we discuss a partial converse of this result. Given an algebra homomorphism $\tau^B: B \to L(A, A \tsr B)$ we will say that $\t^B$ {\it respects multiplication in $A$} if for all $a_1, a_2 \in A$ and $b \in B$, $$\tau^B(b)(a_1 a_2) = (\mu_A \tsr 1_B)(1_A \tsr \ev_{a_2}\tau^B)(\tau^B(b)(a_1)),$$ where $\ev_a: L(A, A \tsr B) \to A \tsr B$ is the map determined by evaluating at $a \in A$. 


\begin{prop} Suppose that $\t^B: B \to L(A, A \tsr B)$ is an algebra homomorphism that respects multiplication in $A$. Define $\tau^A: A \to L(B, A \tsr B)$ by $\tau^A(a)(b) = \tau^B(b)(a)$ for all $a \in A$ and $b \in B$. Then $\tau^A$ is an algebra homomorphism. 

\end{prop}

\begin{proof} Let $a_1, a_2 \in A$ and $b \in B$. Then
\begin{align*}
\tau^A(a_1 a_2)(b) &= \tau^B(b)(a_1 a_2) \\
&= (\mu_A \tsr 1_B)(1_A \tsr \ev_{a_2}\tau^B)(\tau^B(b)(a_1)) \\
&= (\mu_A \tsr 1_B)(1_A \tsr \tau^A(a_2))(\tau^A(a_1)(b)) \\
&= [\tau^A(a_1) \star \tau^A(a_2)](b).
\end{align*}
Therefore $\tau^A(a_1 a_2) = \tau^A(a_1) \star \tau^A(a_2)$, as desired.
\end{proof}

\begin{cor} 
\label{algebra map determines twisting map}
Given an algebra homomorphism $\tau^B: B \to L(A, A \tsr B)$ which respects multiplication in $A$, define a map $\tau: B \tsr A \to A \tsr B$ via $\tau(b \tsr a) = \tau^B(b)(a)$, for  all $a \in A$ and $b \in B$. Then $\tau$ is a twisting map.

\end{cor}

\begin{proof} This follows immediately from Proposition 2.6 in \cite{Cap}.
\end{proof}

The reason for the Corollary is its usefulness for constructing twisting maps: writing down an algebra homomorphism $B \to L(A, A \tsr B)$ which respects multiplication in $A$ can be easier than verifying that some linear map $\tau: B \tsr A \to A \tsr B$ is a twisting map. We now illustrate this point.

Let $A_1$ and $B_1$ be finite-dimensional $\k$-vector spaces. Let $A = T(A_1)$ and $B= T(B_1)$ be the associated tensor algebras with the usual weight grading. Let $\{x_1, \ldots, x_l\}$ be a basis of $A_1$; let $ \{u_1, \ldots, u_m, d_1, \ldots, d_n\}$ be a basis of $B_1$. Let $B_1' = \sp\{u_1, \ldots, u_m\}$, and $B_1'' = \sp\{d_1, \ldots, d_n\}$. Suppose that $t: B_1 \tsr A_1 \to A \tsr B$ is a linear map such that $$t(B_1' \tsr A_1) \subseteq A_1 \tsr B_1' \oplus A_2 \tsr B_0$$and $$t(B_1'' \tsr A_1) \subseteq A_1 \tsr B_1'' \oplus A_0 \tsr\mu_B(B_1 \tsr B_1'').$$

%

\begin{thm}
\label{extendSepTau}
There exists a graded twisting map $\t: B \tsr A \to A \tsr B$ such that $\t|_{B_1 \tsr A_1} = t$.
\end{thm}

In the next section we will prove that such a $\t$ is unique.

\begin{proof}

By Corollary \ref{algebra map determines twisting map}, it suffices to construct an algebra homomorphism $\t^B : B \to L(A, A \tsr B)$ that respects multiplication in $A$. 

We start by defining a linear map $\t^B: B_1 \to L(A, A \tsr B)$. Fix a basis element $u_i \in B'_1$. In order to define $\t^B(u_i) \in L(A, A \tsr B)$ it suffices to define $\t^B(u_i)$ on the canonical monomial basis of $A$ coming from the algebra generating set $\{x_1, \ldots, x_l\}$. First set $\t^B(u_i)(1_A) = 1_A \tsr u_i$ and $\t^B(u_i)(x_j) = t(u_i \tsr x_j)$. Then extend this definition linearly to get a linear map $$\t^B: B_0 \oplus B'_1 \to L(A_{\leq 1}, A_{\leq 2} \tsr (B_0 \oplus B'_1)).$$ With this definition the right hand side of the following equation is well-defined. Set $$\t^B(u_i)(x_{j_1} x_{j_2}) = (\m_A \tsr 1_B)(1_A \tsr \ev_{x_{j_2}} \t^B)(\t^B(u_i)(x_{j_1})).$$ Then, for $k \geq 3$, define, inductively, $$\t^B(u_i)(x_{j_1} \cdots x_{j_k}) = (\m_A \tsr 1_B)(1_A \tsr \ev_{x_{j_k}} \t^B)(\t^B(u_i)(x_{j_1} \cdots x_{j_{k-1}})).$$ Thus we have defined a linear map $\t^B(u_i): A \to A \tsr B$. 

Next we fix a basis element $d_i \in B''_1$. Set $\t^B(d_i)(1_A) = 1_A \tsr d_i$ and $\t^B(d_i)(x_j) = t(d_i \tsr x_j)$. The last definition may be extended linearly to obtain a map $$\t^B: B''_1 \to L(A_{\leq 1}, A_{\leq 1} \tsr B);$$ then since, as noted above, $L(A_{\leq 1}, A_{\leq 1} \tsr B)$ is an associative unital algebra we obtain an algebra homomorphism $$\t^B: T(B''_1) \to L(A_{\leq 1}, A_{\leq 1} \tsr B).$$ Next define $$\t^B(u_k d_i)(x_j) = (1_A \tsr \m_B)(\t^B(u_k) \tsr 1_B)(\t^B(d_i)(x_j)).$$ Now we extend the definition of $\t^B(d_i)$ to the space $A_{\leq 2}$ by setting $$\t^B(d_i)(x_{j_1} x_{j_2}) = (\m_A \tsr 1_B)(\m_A \tsr \ev_{x_{j_2}} \t^B)(\t^B(d_i)(x_{j_1})).$$ Please note that the right hand side of the last equation makes sense as we have defined $\t^B: \m_B(B_1 \tsr B_1) \to L(A_{\geq 1}, A_{\geq 3} \tsr B)$. Suppose $k \geq 3$. Let us assume inductively that for any monomial $w \in T(B_1)_{k-1}$ that $\t^B(w): A_{\leq 1} \to A_{\leq k} \tsr B$ has been defined by multiplicative extension using the fact that $B$ is a free algebra and the algebra structure of $L(A_{\leq k}, A_{\leq k} \tsr B)$. (There is a natural inclusion $L(A_{\leq 1}, A_{\leq k} \tsr B) \to L(A_{\leq k}, A_{\leq k} \tsr B)$ given by extension by $0$.) Then define the map $\t^B(wd_j)$ on the space $A_{\leq 1}$ by  $$\t^B(w d_j) = (1_A \tsr \m_B)(\t^B(w) \tsr 1_B)\t^B(w).$$

Assume that $\tau^B(d_i)(x_{j_1} \cdots x_{j_{k-1}})$ has been defined, then set $$\tau^B(d_i)(x_{j_1} \cdots x_{j_k}) = (\mu_A \tsr 1_B)(1_A \tsr \ev_{x_{j_k}}\tau^B)(\tau^B(d_i)(x_{j_1} \cdots x_{j_{k-1}})).$$ To be clear about why the right hand side makes sense write $$\tau^B(d_i)(x_{j_1} \cdots x_{j_{k-1}}) = (x_{j_1} \cdots x_{j_{k-1}})_{\tau} \tsr (d_i)_{\tau},$$ where $(d_i)_{\tau}$ is a sum of monomials in $T(B_1)$ of degree $\leq k$ and each such monomial ends in an element of $\{d_1, \ldots, d_n\}$. By construction, $\tau^B((d_i)_{\tau})(x_{j_k})$ makes sense. Therefore we have defined $\tau^B(d_i) \in L(A, A \tsr B)$.

We have defined a linear map $\t^B: B_1 \to L(A, A \tsr B)$, so the universal property of the tensor algebra affords an algebra homomorphism $$\t^B: B \to L(A, A \tsr B).$$ Note that the preliminary definitions, for example $\t^B(d_i d_j)$ or $\t^B(u_i d_j)$ as linear maps on $A_{\leq 1}$, are compatible with this definition.

Now we check that $\t^B$ respects multiplication in $A$. Notice that for any element $b_1 \in B_1$ we have defined $$\t^B(b_1)(x_{j_1} \cdots x_{j_k}) = (\m_A \tsr 1_B)(1_A \tsr \ev_{x_{j_k}}\t^B)(\t^B(b_1)(x_{j_1} \cdots x_{j_{k-1}})).$$ It follows easily from the definitions that in fact, for all $b \in B$, $$\t^B(b)(x_{j_1} \cdots x_{j_k}) = (\m_A \tsr 1_B)(1_A \tsr \ev_{x_{j_k}}\t^B)(\t^B(b)(x_{j_1} \cdots x_{j_{k-1}})).$$

Let $b \in B$ be arbitrary. We claim that for all integers $l \geq 2$ and $1 \leq k \leq l$,
$$\tau^B(b)(x_{n_1} \cdots x_{n_k} x_{n_{k+1}} \cdots x_{n_l}) = (\mu_A \tsr 1_B)(1_A \tsr \ev_{x_{n_{k+1}} \cdots x_{n_l}} \tau^B)(\ev_{x_{n_1} \cdots x_{n_k}}\tau^B(b)).$$

As observed above, the case $l = 2$ is true. Suppose that $l \geq 3$ and that the equation holds for all $1 \leq j < l$ and all $1 \leq  k  \leq  l-1$.

On the left hand side of the claimed equation:
\begin{align*}
&\tau^B(b)(x_{n_1} \cdots x_{n_k} x_{n_{k+1}} \cdots x_{n_l})\\ &= (\mu_A \tsr 1_B)(1_A \tsr \ev_{x_{n_l}} \tau^B)(\tau^B(b)(x_{n_1} \cdots x_{n_k} x_{n_{k+1}} \cdots x_{n_{l-1}})) \\
&= (\mu_A \tsr 1_B)(1_A \tsr \ev_{x_{n_l}} \tau^B)(\mu_A \tsr 1_B)(1_A \tsr \ev_{x_{n_{k+1}} \cdots x_{n_{l-1}}} \tau^B)(\ev_{x_{n_1} \cdots x_{n_k}} \tau^B(b)) \\
&=  (\mu_A \tsr 1_B)(\mu_A \tsr 1_A \tsr 1_B)(1_A \tsr 1_A \tsr \ev_{x_{n_l}}\tau^B)(1_A \tsr \ev_{x_{n_{k+1}} \cdots x_{n_{l-1}}} \tau^B)(\ev_{x_{n_1} \cdots x_{n_k}} \tau^B(b)) \\
&= (\mu_A(\mu_A \tsr 1_A) \tsr 1_B)(1_A \tsr 1_A \tsr \ev_{x_{n_l}}\tau^B)(1_A \tsr \ev_{x_{n_{k+1}} \cdots x_{n_{l-1}}} \tau^B)(\ev_{x_{n_1} \cdots x_{n_k}} \tau^B(b)).
\end{align*} 
On the right hand side of the claimed equation:
\begin{align*}
&(\mu_A \tsr 1_B)(1_A \tsr \ev_{x_{n_{k+1}} \cdots x_{n_l}} \tau^B)(\ev_{x_{n_1} \cdots x_{n_k}}\tau^B(b)) \\ &= (\mu_A \tsr 1_B)(1_A \tsr (\mu_A \tsr 1_B)(1_A \tsr \ev_{x_{n_l}}\tau^B)(\ev_{x_{n_{k+1}} \cdots x_{n_{l-1}}}\tau^B))(\ev_{x_{n_1} \cdots x_{n_k}} \tau^B(b)) \\
&= (\mu_A \tsr 1_B)(1_A \tsr (\mu_A \tsr 1_B)(1_A \tsr \ev_{x_{n_l}}\tau^B))(1_A \tsr \ev_{x_{n_{k+1}} \cdots x_{n_{l-1}}}\tau^B)(\ev_{x_{n_1} \cdots x_{n_k}} \tau^B(b))\\
&= (\mu_A \tsr 1_B)(1_A \tsr \mu_A \tsr 1_B)(1_A \tsr 1_A \tsr \ev_{x_{n_l}}\tau^B)(1_A \tsr \ev_{x_{n_{k+1}} \cdots x_{n_{l-1}}}\tau^B)(\ev_{x_{n_1} \cdots x_{n_k}} \tau^B(b)) \\
&= (\mu_A(1_A \tsr \mu_A) \tsr 1_B)(1_A \tsr 1_A \tsr \ev_{x_{n_l}}\tau^B)(1_A \tsr \ev_{x_{n_{k+1}} \cdots x_{n_{l-1}}}\tau^B)(\ev_{x_{n_1} \cdots x_{n_k}} \tau^B(b)). \\
\end{align*}
Since $\mu_A$ is associative, the claim follows.

We conclude that $\t^B$ respects multiplication in $A$. Hence there exists a twisting map $\t: B \tsr A \to A \tsr B$ which extends $t$.

%
%
%

\end{proof}

We conclude with conditions on a graded twisting map that imply the map induces a graded twisting map on graded quotient algebras. The non-graded version appears in \cite[p. 13]{BorowiecMarcinek}.


Let $A$ and $B$ be graded $\k$-algebras. Let $I$ and $J$ be homogeneous ideals in $A$ and $B$ respectively, and let $A'=A/I$ and $B'=B/J$. Let $\pi_A:A\to A'$ and $\pi_B:B\to B'$ denote the canonical projections, which are graded $\k$-algebra homomorphisms. If $X\in \{A,B, A', B'\}$ we denote the multiplication map on $X$ by $\mu_X: X\tensor X\rightarrow X$. Choose graded $\k$-linear splittings $\eta_A:A'\to A$ and $\eta_B:B'\to B$ such that $\pi_A \eta_A = 1_A$ and $\pi_B \eta_B = 1_B$. Assume that $\tau:B\tensor A\rightarrow A\tensor B$ is a graded twisting map such that 
$$\tau( B \tensor I +J \tensor A ) \subset I \tensor B + A \tensor J$$


\begin{thm}
\label{inducedTau}
Make all of the assumptions of the last paragraph. Then the linear map $\tau':B'\tensor A'\rightarrow A'\tensor B'$ given by 
$$\tau' = (\pi_A\tensor \pi_B)\circ \tau\circ (\eta_B\tensor \eta_A)$$
is also a graded twisting map. Moreover, $\pi_A\tensor \pi_B: A\tensor_{\tau} B \to A'\tensor_{\tau'} B'$ is a graded algebra homomorphism.
\end{thm}

We omit the proof and refer the reader to \cite[Theorem, p. 13]{BorowiecMarcinek}. We also remark that it is easy to check that the map $\t'$ does not depend on the choice of the splitting maps $\eta_A$ and $\eta_B$.

\section{Uniqueness of extensions and quadratic relations}
\label{quadratic}

Koszul algebras must be quadratic, so we begin by characterizing quadratic twisted tensor products of quadratic algebras. We show that a twisted tensor product is quadratic if and only if the corresponding twisting map is determined by its lowest-degree component. Initially, it is useful to consider two different notions of ``low-degree determination.'' 

Throughout this section, let $A$ and $B$ be connected, graded, unital $\k$-algebras. Let $\t:B\tsr A\to A\tsr B$ be a graded twisting map. 

\begin{defn}
\label{one-determined definition}
The graded twisting map $\t$ is \emph{one-determined} if for any graded twisting map $\t':B\tsr A\to A\tsr B$ such that $\t_2=\t'_2$ we have $\t=\t'$.
\end{defn}

The unital condition on graded twisting maps implies that degree 2 is the lowest degree in which a difference could occur between $\t$ and $\t'$. 


\begin{defn}
\label{UEP definition}
The graded twisting map $\t$ has \emph{the unique extension property to degree $n$} if, for any linear map  $\t':B\tsr A\to A\tsr B$ that is graded twisting to degree $n$ such that $\t_{i}=\t'_{i}$ for all $i< n$, we have $\t_{n}=\t'_{n}$.

The graded twisting map $\t$ has the \emph{unique extension property} if $\t$ has the unique extension property to degree $n$ for all $n\ge 3$.
\end{defn}

It is obvious that a graded twisting map $\t$ with the unique extension property is one-determined. 

Example \ref{1detNotUEP} shows that the converse is false, though we do not know of a counterexample where both $A$ and $B$ are quadratic. 

When $\t$ is a graded twisting map such that the ideal $I_{\t}$ is generated by (homogeneous) quadratic elements, we say that $A\tsr_{\t} B$ is \emph{$\t$-quadratic}. (The definition of $I_{\t}$ occurs immediately preceding Proposition \ref{presentation}.) Now we characterize graded twisting maps which have the unique extension property as follows. 



\begin{thm}
\label{UEPquadratic}
A graded twisting map $\t:B\tsr A\to A\tsr B$ has the unique extension property if and only if $A\tsr_{\t} B$ is $\t$-quadratic.
\end{thm}

\begin{proof}

Suppose $\t$ has the unique extension property. Assume inductively that $I^{n}_{\t}=I^2_{\t}$. 
Let $b\tsr a\in (B\tsr A)_{n+1}$ where $a\in A_+$ and $b\in B_+$. Let $r=b\tsr a-\t(b\tsr a)\in A\ast B$. By Lemma \ref{TMEP}, there exists $y\in (B\tsr A\tsr A)\oplus (B\tsr B\tsr A)$ such that $\d(y)=b\tsr a$ and $R(y) = \t(b\tsr a)$. We can assume $y=\sum b_i\tsr a_i\tsr a'_i + \sum b'_j\tsr b''_j\tsr a''_j$ where
$a_i, a'_i, a''_j\in A_+$ and $b_i, b'_j, b''_j\in B_+$.  

Letting $\t'= \t-i_A\t^A-i_B\t^B$, we perform the following calculation in the free product $A\ast B$.
\begin{align*}
r&=b\tsr a- R(y) = \d(y)-R(y)\\
&=\sum b_i\tsr a_ia'_i - \t^B(b_i\tsr a_i)\tsr a'_i - R(b_i\tsr a_i\tsr a'_i) \\
&\qquad +\sum b'_jb''_j\tsr a''_j - b'_j\tsr \t^A(b''_j\tsr a''_j) - R(b'_j\tsr b''_j\tsr a''_j)\\
&= \sum \left(b_i\tsr a_i - \t^B(b_i\tsr a_i) - \t^A(b_i\tsr a_i)\right)\tsr a'_i - (\mu_A\tsr 1)(1_A\tsr \t)(\t'(b_i\tsr a_i)\tsr a'_i)\\
&\qquad +  \sum b'_j\tsr\left(b''_j\tsr a''_j - \t^A(b''_j\tsr a''_j) - \t^B(b''_j\tsr a''_j)\right) - (1_A\tsr \mu_B)(\t\tsr 1_B)(b'_j\tsr \t'(b''_j\tsr a''_j))
\end{align*}

Now, $$\t'(b_i\tsr a_i)\tsr a'_i - (1_A\tsr \t)(\t'(b_i\tsr a_i)\tsr a'_i)$$ and 
$$b'_j\tsr \t'(b''_j\tsr a''_j) - (\t\tsr 1_B)(b'_j\tsr \t'(b''_j\tsr a''_j))$$ are both elements of $I^n_{\t}=I^2_{\t}$. Adding these relations to $r$, we see that 
$$r- \sum (b_i\tsr a_i-\t(b_i\tsr a_i))\tsr a'_i - \sum b'_j\tsr (b''_j\tsr a''_j-\t(b''_j\tsr a''_j))\in I_{\t}^2.$$
Since none of $a_i, a'_i, a''_j$ or $b_i, b'_j, b''_j$ are units, $b_i\tsr a_i-\t(b_i\tsr a_i), b''_j\tsr a''_j-\t(b''_j\tsr a''_j) \in I^n_{\t}=I^2_{\t}$, hence $r\in I^2_{\t}$. We conclude that $A \tsr_{t} B$ is $\t$-quadratic.

Conversely, let $\t:B\tsr A\to A\tsr B$ be a graded twisting map and assume $I_{\t}$ is generated as an ideal of $A\ast B$ by $(I_{\t})_2$. We must prove that for any $n\ge 3$ and any linear map $\t':B\tsr A\to A\tsr B$ that is graded twisting to degree $n$ and that satisfies $\t'_i=\t_i$ for every $i< n$, we have $\t'_{n}=\t_{n}$. 

Let $I_{\t}^n$ and $I_{\t'}^n$ be the ideals generated by all $\t$-relations (resp. $\t'$-relations) of degree at most $n$. Note that $I_{\t}^n=I_{\t}^2=I_{\t}$ since $A\tsr_{\t} B$ is $\tau$-quadratic. Since $\t'_2=\t_2$, we have $I_{\t}^2=I_{\t'}^2$ and hence $I_{\t}= I_{\t'}^2\subseteq I_{\t'}^n$. 

Let $\gamma\in (B\tsr A)_{n}$ be arbitrary. Let $w=\gamma-\t(\gamma)$ and $w'=\gamma-\t'(\gamma)$. Then $w\in I_{\t}\subseteq I_{\t'}^n$. By definition, $w'\in I_{\t'}^n$. Thus $\t(\gamma)-\t'(\gamma)=w'-w\in I_{\t'}^n$. By Proposition \ref{presentation}, the canonical map $(A\ast B)/I_{\t'}^n \to (A\tsr B)/(A\tsr B)_{>n}$ is an isomorphism in degree $n$ since $\t'$ is graded twisting to degree $n$. Since $w'-w\in (A\tsr B)_{n}$, it follows that $w'=w$ and hence $\t_{n}=\t'_{n}$. Thus $\t$ has the unique extension property.

\end{proof}

\begin{cor}
\label{quadTTP}
Let $\t:B\tsr A\to A\tsr B$ be a graded twisting map with the unique extension property. If $A$ and $B$ are quadratic $\k$-algebras, then $A\tsr_{\t} B$ is quadratic.
\end{cor}

Example \ref{Ugh3} shows that the converse of Corollary \ref{quadTTP} is false even under the stronger assumption that $A\tsr_{\t} B$ is a Koszul algebra. However we do have the following positive result.

\begin{prop}
Let $A$ and $B$ be quadratic $\k$-algebras and let $\t:B\tsr A\to A\tsr B$ be a graded twisting map such that $A\tsr_{\t} B$ is a quadratic algebra. Then $\t$ has the unique extension property.
\end{prop}

\begin{proof}
Write $A = T(A_1)/I_A$ and $B = T(B_1)/I_B$. Let $I_A^2 = \la (I_A)_2 \ra$ be the ideal of $T(A_1)$ generated by the quadratic part of $I_A$. Similarly define $I_B^2 = \la (I_B)_2 \ra$. The assumption that $A$ and $B$ are quadratic means that $I_A^2 = I_A$ and $I_B^2 = I_B$. Consider the tensor algebra $T(A_1 \oplus B_1)$. There are canonical inclusions of $T(A_1)$ and $T(B_1)$ into $T(A_1 \oplus B_1)$. Let $J_A$ and $J^2_A$ denote the ideals of $T(A_1 \oplus B_1)$ generated by $I_A$ and $I^2_A$, respectively. Similarly define $J_B$ and $J^2_B$. 

Choose graded $\k$-linear sections $\eta_A: A \to T(A_1)$ and $\eta_B: B \to T(B_1)$ of the canonical algebra maps $\pi_A :T(A_1) \to A$ and $\pi_B :T(B_1) \to B$, respectively. Define  ideals of $T(A_1 \oplus B_1)$ by $$J_{\t} = \la (\eta_B \tsr \eta_A)(b \tsr a - a_{\t} \tsr b_{\t}) : a \in A, \ b \in B\ra$$ and $$J^2_{\t} = \la (\eta_B \tsr \eta_A)(b \tsr a - a_{\t} \tsr b_{\t}) : a \in A_i, \ b \in B_j, \ i+j \leq 2\ra.$$ 

By Proposition \ref{presentation}, we know that $\displaystyle{A \tsr_{\t} B \cong (A \ast B)/I_{\t}}$. Then we have the following isomorphisms
\begin{align*}
(A \ast B)/I_{\t} &\cong A \tsr_{\t} B \\
&\cong \frac{T(A_1 \oplus B_1)}{J_A + J_B + J_{\t}} \\
&\cong \frac{T(A_1 \oplus B_1)}{\la (J_A+J_B + J_{\t})_2 \ra} \\
&\cong \frac{T(A_1 \oplus B_1)}{J^2_A + J^2_B + J^2_{\t}} \\
&\cong (A \ast B)/I^2_{\t}.
\end{align*}
We have used the hypothesis that $A \tsr_{\t} B$ is quadratic on the third isomorphism, and the assumption that $A$ and $B$ are quadratic on the last isomorphism. The fourth isomorphism is actually an equality since it is easy to check that $\la (J_A+J_B + J_{\t})_2 \ra = J^2_A + J^2_B + J^2_{\t}$. Now, since $I^2_{\t} \subseteq I_{\t}$, we conclude that $I^2_{\t} = I_{\t}$. Hence $A \tsr_{\t} B$ is $\t$-quadratic, from which it follows, by Theorem \ref{UEPquadratic}, that $\t$ has the unique extension property.
\end{proof}

To conclude this section we will show that a certain large class of two-sided graded twisting maps has the unique extension property. As promised, this will show that the graded twisting map constructed in Theorem \ref{extendSepTau} is unique. Furthermore, we will study the Koszul property for a subclass of this class in the next section.

For the rest of this section assume that $A$ and $B$ are one-generated $\k$-algebras. Suppose that $\t: B \tsr A \to A \tsr B$ is a graded twisting map and that there is a vector space decomposition $B_1 = B_1' \oplus B_1''$ such that $$\t(B_1' \tsr A_1) \subseteq A_2 \tsr B_0 \oplus A_1 \tsr B_1,$$
$$\t(B_1'' \tsr A_1) \subseteq A_1 \tsr B_1 \oplus A_0 \tsr \m_B(B_1 \tsr B_1'' \oplus B_1' \tsr  B_1').$$

\begin{thm}
\label{generalized separable maps have UEP}
Assume the hypotheses of the last paragraph. Then $\t$ has the unique extension property.
\end{thm}

\begin{proof} Let $n \geq 2$ and let $\t': B \tsr A \to A \tsr B$ be a linear map that is graded twisting to degree $n$ such that $\t_{\leq 2} = \t'_{\leq 2}$. We first show that $\t|_{(B_1 \tsr A)_n} = \t'|_{(B_1 \tsr A)_n}$.

The case $n=2$ is true by assumption. Assume that $n \geq 3$ and that $\t|_{(B_1 \tsr A)_{<n}} = \t'|_{(B_1 \tsr A)_{<n}}$.

First, let $u \tsr x \in (B'_1 \tsr A)_n$. Write $x = \mu_A(x_1 \tsr x_2)$, where $\deg(x_1) = 1$ (note that we are using Sweedler-type notation here since actually $x$ may be a sum of terms of the form $\mu_A(x_1 \tsr x_2)$). Then
\begin{align*}
\t(u \tsr x) &= \t(1_B \tsr \mu_A)(u \tsr x_1 \tsr x_2) \\
&= (\m_A \tsr 1_B)(1_A \tsr \t)(\t_2 \tsr 1_A)(u \tsr x_1 \tsr x_2) \\
&= (\m_A \tsr 1_B)(1_A \tsr \t)(\t'_2 \tsr 1_A)(u \tsr x_1 \tsr x_2).
\end{align*}
By assumption $\t'_2(u \tsr x_1) = (x_1)_{\t} \tsr u_{\t} \in A_2 \tsr B_0 \oplus A_1 \tsr B_1$. Since $\deg(x_2)< \deg(x)$ it follows by induction that $\t(u_{\t} \tsr x_2) = \t'(u_{\t} \tsr x_2)$. Therefore $\t(u \tsr x) = \t'(u \tsr x)$.

Second, let $d \tsr x \in (B''_1 \tsr A)_n$.  We start by noting that if we write $\t(d \tsr x) = x_{\t} \tsr d_{\t} \in A \tsr B$, then $\deg(x_{\t}) \leq \deg(x)$. Write $x = \m_A(x_1 \tsr x_2)$, where $\deg(x_1) = 1$. Then 
\begin{align*}
\t(d \tsr x) &= \t(1_B \tsr \m_A)(d \tsr x_1 \tsr x_2) \\
&= (\m_A \tsr 1_B)(1_A \tsr \t)(\t_2 \tsr 1_A)(d \tsr x_1 \tsr x_2) \\
&= (\m_A \tsr 1_B)(1_A \tsr \t)(\t'_2 \tsr 1_A)(d \tsr x_1 \tsr x_2). 
\end{align*}
By assumption $\t'_2(d \tsr x_1) = (x_1)_{\t} \tsr d_{\t} \in A_1 \tsr B_1 \oplus A_0 \tsr \m_B(B_1 \tsr B_1'' \oplus B_1' \tsr B_1')$. Therefore we may write $d_{\t}$ as a sum of terms in $B_1$, $\m_B(B_1 \tsr B''_1)$, or $\m_B(B'_1 \tsr B'_1)$. We address each of these three possibilities. If $b \in B_1$ is a term in $d_{\t}$, then using the inductive assumption, $\t(b \tsr x_2) = \t'(b \tsr x_2)$. If $\m_B(b \tsr d') \in \m_B(B_1 \tsr B''_1)$ is a term in $d_{\t}$, then, using the inductive assumption and the result of the previous paragraph, we see that $\t(\m_B(b \tsr d') \tsr x_2) = \t'(\m_B(b \tsr d') \tsr x_2)$. Finally if $\m_B(u_1 \tsr u_2) \in \m_B(B'_1 \tsr B'_1)$ is a term in $d_{\t}$, then two applications of the result in the previous paragraph yields $\t(\m_B(u_1 \tsr u_2) \tsr x_2) = \t'(\m_B(u_1 \tsr u_2) \tsr x_2)$. Therefore it follows that $\t(d \tsr x) = \t'(d \tsr x)$.

Finally, let $n \geq 3$. Fix a tensor $b \tsr a \in (B_+ \tsr A_+)_{n}$; also assume, without loss of generality, that $b$ is a monomial (with respect to the algebra generating set $B_1$). We induct on $k = \deg(b)$. The base of the induction, where $k = 1$, is taken care of above.

 Let $k \geq 2$ and write $b = b_1 b_2$, where $\deg(b_2) = 1$. Assume inductively that $\t|_{(B_{<k} \tsr A)_{\leq n}} = \t'|_{(B_{<k} \tsr A)_{\leq n}}$. Then
\begin{align*}
\t(b \tsr a) &= \t(\m_B \tsr 1_A)(b_1 \tsr b_2 \tsr a) \\
&= (1_A \tsr \m_B)(\t \tsr 1_B)(1_B \tsr \t)(b_1 \tsr b_2 \tsr a) \\
&= (1_A \tsr \m_B)(\t' \tsr 1_B)(1_B \tsr \t')(b_1 \tsr b_2 \tsr a) \\
&= \t'(b \tsr a).
\end{align*}

\end{proof}


\section{Koszul twisted tensor products}
\label{Koszul twisted tensor products}

%

Recall that a graded twisting map $\t:B\tsr A\to A\tsr B$ is  called \emph{one-sided} if 
$$\t(B_+\tsr A_+)\subseteq (A_+\tsr B_0) \oplus (A_+\tsr B_+)$$ or 
$$\t(B_+\tsr A_+)\subseteq  (A_+\tsr B_+) \oplus (A_0\tsr B_+).$$
The most common examples of algebras constructed from graded twisting maps: Ore extensions, crossed product algebras, or $H$-module algebra smash products, where $H$ is a graded Hopf algebra, come from one-sided twisting maps. 


The following simple observation reveals why one-sided twisting maps more readily permit transfer of structure from $A$ and $B$ to $A\tsr_{\t} B$ than (two-sided) graded twisting maps in general.

\begin{prop}
\label{normalSubalgebra}
If $\t:B\tsr A\to A\tsr B$ is one-sided, then either $A$ or $B$ is a normal subalgebra of $A\tsr_{\t} B$.
\end{prop}

\begin{proof}
This follows immediately from Proposition \ref{presentation}.

\end{proof}

Proposition \ref{normalSubalgebra} suggests that twisted tensor products arising from one-sided twisting maps can be fruitfully studied using standard techniques concerning normal subalgebras. Such is the case with Theorem \ref{oneSidedKoszul} below. Before addressing the Koszul property, we examine one-sided twisting maps in the framework of Section \ref{quadratic}.

\begin{prop}
\label{oneSidedUEP}
Let $A$ and $B$ be one-generated graded $\k$-algebras. If $\t:B\tsr A\to A\tsr B$ is a one-sided graded twisting map, then $\t$ has the unique extension property. 
\end{prop}


%

\begin{proof}
This result follows immediately from Theorem \ref{generalized separable maps have UEP}.

\end{proof}

\begin{thm}
\label{oneSidedKoszul}
Let $\t:B\tsr A\to A\tsr B$ be a one-sided graded twisting map. 
\begin{enumerate}
\item If $A$ and $B$ are quadratic $\k$-algebras, then $A\tsr_{\t} B$ is quadratic. 
\item If $A$ and $B$ are Koszul $\k$-algebras, then $A\tsr_{\t} B$ is Koszul.
\end{enumerate}
\end{thm}

\begin{proof}
Statement (1) is an immediate consequence of Proposition \ref{oneSidedUEP} and the last statement of Proposition \ref{presentation}.

Assume that $A$ and $B$ are Koszul algebras and $$\t(B_+\tsr A_+)\subseteq  (A_+\tsr B_+) \oplus (A_0\tsr B_+).$$  Let $C = A \tsr_{\t} B$. Then the map $i_B: B \to C$ given by $i_B(b) = 1 \tsr b$ for all $b \in B$ is an injective homomorphism of algebras. Moreover, $C$ is a free right $B$-module. The twisting map condition ensures that $i_B(B_1)C_1 \subseteq C_1 i_B(B_1)$. Finally note that $C/(i_B(B_1)C) \cong A$ as algebras. The case where $$\t(B_+\tsr A_+)\subseteq (A_+\tsr B_0) \oplus (A_+\tsr B_+)$$ is analogous. Now Statement (2) follows from Statement (1) and Corollary \ref{one-sided Koszul}.

\end{proof}


We assume that (2) of Theorem \ref{oneSidedKoszul} is well-known, but we were not able to find the statement in this generality in the literature.  
%
%


Absent the requirement that at least one of $A$ and $B$ is a normal subalgebra of $A\tsr_{\t} B$,  questions about the structure of $A\tsr_{\t} B$ become much more difficult to answer. Even constructing graded twisting maps that are not one-sided is far from straightforward. We begin with an example that shows that a twisted tensor product of Koszul algebras need not be Koszul, nor even quadratic. The example also shows: the twisted tensor product of two algebras of finite global dimension need not have finite global dimension, and the class of Artin-Schelter regular algebras is not closed under taking twisted tensor products.


\begin{ex}
\label{non-Koszul example}
Let $A = \k[x]$, $B = \k[y]$. Define a $\k$-linear map $\t: B \tsr A \to A \tsr B$ by $$\t(y^i \tsr x^j) = \begin{cases} x^j \tsr y^i &\text{ if } i \text{ or } j \text{ is even} \\
x^{j+1} \tsr y^{i-1} - x^j \tsr y^i + x^{j-1} \tsr y^{i+1} &\text{ if } i \text{ and } j \text{ are both odd}.
\end{cases}$$
We make the following claims about $\t$:
\begin{enumerate}
\item $\t$ is a graded twisting map,
\item $\t$ does not have the unique extension property to degree 3,
\item $\t$ has the unique extension property to degree $n$, for all $n \geq 4$.
\end{enumerate}
For (1) we use Theorem \ref{simpler extension condition}. First of all, in degree 3, we have
\begin{align*}
(\m_A \tsr 1_B)(1_A \tsr \t)(\t \tsr 1_A)(y \tsr x \tsr x) &= (\m_A \tsr 1_B)(1_A \tsr \t)((x^2 \tsr 1_B - x \tsr y + 1_A \tsr y^2) \tsr x)  \\
&= (\m_A \tsr 1_B)(x^2 \tsr x \tsr 1_B - x \tsr x^2 \tsr 1_B + x \tsr x \tsr y \\
& \qquad \qquad \ \ \ \ \ \ \  - x \tsr 1_A \tsr y^2+ 1_A \tsr x \tsr y^2) \\
&= x^2 \tsr y \\
&= \t(y \tsr x^2).
\end{align*}
Similiarly, $(1_A \tsr \m_B)(\t \tsr 1_B)(1_B \tsr \t)(y \tsr y \tsr x) = \t(y^2 \tsr x)$. Now assume $\t_{\leq n}$ is twisting to degree $n$, for $n \geq 3$. We check condition (2) of Theorem \ref{simpler extension condition} and leave the check of condition (3) of Theorem \ref{simpler extension condition} to the reader.

Suppose that $i+j = n+1$ and $i, j \geq 1$. 

{\bf Case 1:} $i$ is even.

Then
\begin{align*}
(\m_A \tsr 1_B)(1_A \tsr \t)(\t \tsr 1_A)(y^i \tsr x^j \tsr x) &= (\m_A \tsr 1_B)(1_A \tsr \t)(x^j \tsr y^i \tsr x) \\
&= (\m_A \tsr 1_B)(x^j \tsr x \tsr y^i) \\
&= x^{j+1} \tsr y^i \\
&= \t(y^i \tsr x^{j+1}).
\end{align*}

{\bf Case 2:} $i$ and $j$ are both odd.

Then 
\begin{align*}
&(\m_A \tsr 1_B)(1_A \tsr \t)(\t \tsr 1_A)(y^i \tsr x^j \tsr x) \\
 &= (\m_A \tsr 1_B)(1_A \tsr \t)((x^{j+1} \tsr y^{i-1} - x^j \tsr y^i+ x^{j-1} \tsr y^{i+1}) \tsr x) \\
&= (\m_A \tsr 1_B)(x^{j+1} \tsr x \tsr y^{i-1} - x^j \tsr(x^2 \tsr y^{i-1}-x \tsr y^i+1_A \tsr y^{i+1})+x^{j-1} \tsr x \tsr y^{i+1})) \\
&= x^{j+1} \tsr y^i \\
&= \t(y^i \tsr x^{j+1}).
\end{align*}

The case where $i$ is odd and $j$ is even is left to the reader. We conclude that $\t$ is a twisting map.

For (2), let $\a, \b, \gamma, \d \in \k$ be arbitrary. Define a graded $\k$-linear map $\t': B \tsr A \to A \tsr B$ by insisting that $\t'$ satisfy the unital twisting conditions, $\t'(y \tsr x) = x^2 \tsr 1_B - x \tsr y + 1_A \tsr y^2$, $\t'(y \tsr x^2) = \a x^3 \tsr 1_B + \b x^2 \tsr y + \gamma x \tsr y^2 + \d 1_A \tsr y^3$, $\t'(y^2 \tsr x) = \a x^3 \tsr 1_B + (\b-1) x^2 \tsr y + (\gamma+1) x \tsr y^2 + \d 1_A \tsr y^3$, and $\t'$ on $(B_+ \tsr A_+)_{\geq 4}$ is defined arbitrarily. It is straightforward to check that $\t'$ is a twisting map to degree 3. We conclude that $\t$ does not have the unique extension property to degree 3. Consequently, by Theorem \ref{UEPquadratic}, $A \tsr_{\t} B$ is not quadratic.

For (3), assume that $\t': B \tsr A \to A \tsr B$ is graded twisting to degree $n$, $n \geq 4$, and $\t'_i = \t_i$ for all $i < n$. We have to show that $\t'_n = \t_n$. Let $i, j \geq 1$ and $i+j = n$. If $i = 1$, then $j \geq 3$, so 
\begin{align*}
\t'(y \tsr x^j) &= \t'(1_B \tsr \m_A)(y \tsr x^2 \tsr x^{j-2}) \\
&= (\m_A \tsr 1_B)(1_A \tsr \t')(\t' \tsr 1_A)(y \tsr x^2 \tsr x^{j-2}) \\
&= \t(y \tsr x^j).
\end{align*}

If $i \geq 2$, then 

\begin{align*}
\t'(y^i \tsr x^j) &= \t'(\m_B \tsr 1_A)(y^{i-2} \tsr y^2 \tsr x^{j}) \\
&= (1_A \tsr \m_B)\t' \tsr 1_A)(1_A \tsr \t')(y^{i-2} \tsr y^2 \tsr x^{j}) \\
&= \t(y \tsr x^j).
\end{align*} 

We conclude that $\t$ has the unique extension property to degree $n$.

Finally, it is easy to prove that $A \tsr_{\t} B$ is isomorphic to the algebra $\k \la x, y \ra/ \la x^2, y^2x-xy^2 \ra$. This algebra has infinite global dimension, so $A \tsr_{\t} B$ is not Artin-Schelter regular. 


\end{ex}

%

In Example \ref{non-Koszul example} the twisted tensor product of free algebras is not Koszul because the graded twisting map $\t$ fails to have the unique extension property. When $A$ and $B$ are free algebras, the unique extension property completely determines Koszulity.

\begin{prop}
\label{freeKoszul}
If $A$ and $B$ are one-generated free algebras and $\t:B\tsr A\to A\tsr B$ is a graded twisting map, then $A\tsr_{\t} B$ is a Koszul algebra whenever it is quadratic.
\end{prop}

\begin{proof}
Suppose that $A = \k \la x_1, \ldots, x_m \ra$ and $B = \k \la y_1, \ldots, y_n \ra$ are one-generated free algebras. Assume that $\t:B\tsr A\to A\tsr B$ is a graded twisting map. Let $C = A \tsr_{\t} B$ and suppose that $C$ is quadratic. The Hilbert series of $A$ and $B$ are respectively, $h_A(t) = (1-mt)^{-1}$ and $h_B(t) = (1-nt)^{-1}$, so $$h_C(t) = (1-(m+n)t+mnt^2)^{-1}.$$ That $C$ is Koszul follows immediately from \cite{PP} Chapter 2, Proposition 2.3.
\end{proof}

We note that under the assumption that $A$ and $B$ are free, Theorem \ref{UEPquadratic} implies $A\tsr_{\t} B$ is quadratic if and only if $\t$ has the unique extension property. In the very simplest case where $A$ and $B$ are free on a single generator, almost all graded twisting maps have the unique extension property; see Section \ref{abc}.

We do not know of an example where $A$ and $B$ are Koszul, and $A\tsr_{\t} B$ is quadratic but not Koszul. We would not be surprised if such examples exist. The graded twisting map $\t$ for such an example would necessarily be two-sided.

In the remainder of this section we introduce a large class of graded two-sided twisting maps whose associated twisted tensor products are Koszul algebras. This class contains algebras whose Koszulity cannot be explained by Theorem \ref{oneSidedKoszul}; see Example \ref{2sidedTTP}.

Let $A$ and $B$ be one-generated $\k$-algebras. Recall that for a vector space $V$, the tensor algebra generated by $V$ is denoted $T(V)$. Since $A$ and $B$ are one-generated, there are canonical projections $\pi_A:T(A_1)\to A$ and $\pi_B:T(B_1)\to B$. 




\begin{defn}
\label{separable definition}
A graded twisting map $\tau:B\tsr A\to A\tsr B$ is called {\it separable} if there exists a decomposition $B_1=B_1'\oplus B_1''$ such that
\begin{align*}
\tau(B_1'\tsr A_1) &\subset A_2\tsr B_0\oplus A_1\tsr B_1'\\
\tau(B_1''\tsr A_1) &\subset A_1\tsr B_1''\oplus A_0\tsr \mu_B(B_1\tsr B_1'').
\end{align*}
\end{defn}

We remark that the class of separable twisting maps is a subclass of the twisting maps introduced in the paragraph prior to Theorem \ref{generalized separable maps have UEP}. Hence if $\t:B \tsr A \to A \tsr B$ is separable, then $\t$ has the unique extension property, and so $A \tsr_{\t} B$ is quadratic when $A$ and $B$ are quadratic.

One approach to proving an algebra is Koszul that is useful in many situations makes use of semigroup filtrations. In the context of separable graded twisting maps, it is natural to consider the following $\N^3$ filtration - indeed this filtration motivates the definition of separable graded twisting map.

Let $\G=\N^3$. Then $\G$ is a commutative monoid under componentwise addition, and we  equip $\G$ with a grading by total degree; that is, we let $g:\G\to \N$ be the grading homomorphism defined by $g(a,b,c)=a+b+c$. Ordering the fibers $g^{-1}(n)$ lexicographically, with 
$$(0,0,1)<(0,1,0)<(1,0,0)\quad\text{ in } g^{-1}(1)$$
determines the structure of a graded, ordered monoid on $\G$. In particular, if $\a, \b\in g^{-1}(n)$ and $\gamma\in g^{-1}(m)$, then $\a<\b$ implies $\a+\gamma<\b+\gamma$. 

We define a one-generated, $\G$-valued filtration $F_{\bullet}$ on $A\tsr_{\t} B$ as follows:
$$F_{(0,0,0)} = A_0\tsr B_0\qquad F_{(0,0,1)} = A_0\tsr B_1'' + F_{(0,0,0)} $$ 
$$F_{(0,1,0)} = A_1\tsr B_0 + F_{(0,0,1)}\qquad F_{(1,0,0)} = A_0\tsr B_1'+F_{(0,1,0)}$$
and for all $\a\in\G$ such that $g(\a)>1$,
$$F_{\a} = \sum_{i_1+\cdots+i_k=\a} F_{i_1}\cdot F_{i_2}\cdots F_{i_k}$$

Note that by restricting our attention to the subalgebra $B$ we obtain a filtration by the commutative monoid $\N^2$, viewed as the submonoid of $\Gamma$ generated by $(1,0,0)$ and $(0,0,1)$. We denote this filtration by $F^B$.

With Theorem \ref{oneSidedKoszul} in mind, we address the question of when the associated graded algebra of $A \tsr_{\t} B$ with respect to the filtration $F$ is a graded twisted tensor product. 

We begin with several lemmas. In what follows we will use $1$ to denote either of the identity maps $1_A$ and $1_B$. It is clear from context which one we mean.


\begin{lemma}
\label{extendingR}
Let $A$ and $B$ be quadratic algebras with quadratic relation spaces $I_2$ and $J_2$, respectively. If $R:B_1\tsr A_1\to A_1\tsr B_1$ is a linear map such that
$$(1\tsr R)(R\tsr 1)(B_1\tsr I_2)\subseteq I_2\tsr B_1$$ and
$$(R\tsr 1)(1\tsr R)(J_2\tsr A_1)\subseteq A_1\tsr J_2$$
then $R$ extends uniquely to a graded twisting map $\widetilde{R}:B\tsr A\to A\tsr B$
\end{lemma}

\begin{proof}
By Theorem \ref{extendSepTau}, $R$ uniquely determines a graded twisting map between free algebras $\widehat{R}:T(B_1)\tsr T(A_1)\to T(A_1)\tsr T(B_1)$. Since $R(B_1\tsr A_1)\subset A_1\tsr B_1$, it follows by induction on total degree that $\widehat{R}(T_n(B_1)\tsr T_m(A_1))\subseteq T_m(A_1)\tsr T_n(B_1)$ for all $m, n\in \N$. Furthermore, the assumption 
$$(1\tsr R)(R\tsr 1)(B_1\tsr I_2)\subseteq I_2\tsr B_1$$  implies
$$(1\tsr \widehat{R})(\widehat{R}\tsr 1)(T(B_1)\tsr \la I_2 \ra)\subseteq \la I_2 \ra \tsr T(B_1)$$ and
similarly $$(R\tsr 1)(1\tsr R)(J_2\tsr A_1)\subseteq A_1\tsr J_2$$
implies $$(\widehat{R}\tsr 1)(1\tsr \widehat{R})(\la J_2 \ra \tsr T(A_1))\subseteq T(A_1)\tsr \la J_2 \ra.$$ 
By Theorem \ref{inducedTau}, $\widehat{R}$ induces a graded twisting map $\widetilde{R}:B\tsr A\to A\tsr B$.
\end{proof}

Let $I_2=\ker(\m_A|_{A_1\tsr A_1})$ and $J_2=\ker(\m_B|_{B_1\tsr B_1})$, so if $A$ and $B$ are quadratic, then $A\cong T(A_1)/\la I_2 \ra$ and $B\cong T(B_1)/\la J_2 \ra$. 
Define $R:B_1\tsr A_1\to A_1\tsr B_1$ by $R=\pi_{A_1\tsr B_1}\circ \t|_{B_1\tsr A_1}$. Also define $\t_A: B_1 \tsr A_1 \to A_2 \tsr B_0$ by $\t_A = \pi_{A_2 \tsr B_0} \circ \t|_{B_1\tsr A_1}$ and $\t_B: B_1 \tsr A_1 \to A_0 \tsr B_2$ by $\t_B = \pi_{A_0 \tsr B_2} \circ \t|_{B_1\tsr A_1}$.

\begin{lemma}
\label{firstone}
If $\t:B\tsr A\to A\tsr B$ is a separable graded twisting map, then 
$$(1\tsr R)(R\tsr 1)(B_1'\tsr I_2)\subseteq I_2 \tsr B_1'.$$ 
Furthermore, if 
$$\pi_{A_0\tsr A_2\tsr B_1}(1\tsr \t)(\t_B\tsr 1)(B_1\tsr I_2)=0$$
then
$$(1\tsr R)(R\tsr 1)(B_1''\tsr I_2)\subseteq I_2 \tsr B_1''.$$ 
\end{lemma}

It is obvious from the definitions that the additional condition is equivalent to 
$$\pi_{A_0\tsr A_2\tsr B_1}(1\tsr \t)(\t_B\tsr 1)(B_1''\tsr I_2)=0.$$

\begin{proof}
For all $u\in B_1'$ and $x,y\in A_1$ we have
\begin{align*}
(1\tsr \t)(\t\tsr 1)(u\tsr x\tsr y)&=(1\tsr \t)(\t_A(u\tsr x)\tsr y+R(u\tsr x)\tsr y)\\
&= (1\tsr \t)(\t_A(u\tsr x)\tsr y)+x_R\tsr \t_A(u_R\tsr y) + x_R\tsr R(u_R\tsr y).
\end{align*}
Since the first two terms are in $A_+\tsr A_+\tsr B_0$, we have
$$(1\tsr R)(R\tsr 1)(B_1'\tsr I_2) =\pi_{A_1\tsr A_1\tsr B_1'}(1\tsr \t)(\t\tsr 1)(B_1'\tsr I_2).$$
Since $(1\tsr \t)(\t\tsr 1)(B_1'\tsr I_2)\subset I_2 \tsr B,$ it follows that
$$(1\tsr R)(R\tsr 1)(B_1'\tsr I_2)\subseteq I_2 \tsr B_1'.$$

Now, for all $d\in B_1''$ and $x,y\in A_1$ we have
\begin{align*}
(1\tsr \t)(\t\tsr 1)(d\tsr x\tsr y)&=(1\tsr \t)(\t_B(d\tsr x)\tsr y + R(d\tsr x)\tsr y)\\
&=(1\tsr \t)(\t_B(d\tsr x)\tsr y) + x_R\tsr \t_B(d_R\tsr y) + x_R\tsr R(d_R\tsr y).
\end{align*}
Assume 
$$\pi_{A_0\tsr A_2\tsr B_1}(1\tsr \t)(\t_B\tsr 1)(B_1\tsr I_2)=0.$$
Then for any element $d\tsr \sigma\in B_1''\tsr I_2$, the first two terms in the last line above are in 
$(A\tsr A)_{\le 1}\tsr B_{\ge 2}$ and hence
$$(1\tsr R)(R\tsr 1)(B_1''\tsr I_2)=\pi_{A_1\tsr A_1\tsr B_1''}(1\tsr \t)(\t\tsr 1)(B_1''\tsr I_2).$$
As before, since $(1\tsr \t)(\t\tsr 1)(B_1''\tsr I_2)\subset I_2\tsr B,$ we conclude that
$$(1\tsr R)(R\tsr 1)(B_1''\tsr I_2)\subseteq I_2 \tsr B_1''.$$ 
\end{proof}

When $(1\tsr R)(R\tsr 1)(B_1\tsr I_2)\subseteq I_2\tsr B_1,$ the extension of $R$ to $B_1\tsr A_2$ given by 
$$R(b\tsr a) = (\mu_A\tsr 1)(1\tsr R)(R\tsr 1)(b\tsr \hat{a}),$$
where $\hat{a}\in\mu_A^{-1}(a)$ is any preimage of $a$ in $A_1 \tsr A_1$, is well-defined. Next, we establish analogous conditions under which $R$ extends to $B_2\tsr A_1$. This will be enough to ensure $R$ extends to a graded twisting map on all of $B\tsr A$.

\begin{lemma}
\label{secondone}
If $\t:B\tsr A\to A\tsr B$ is a separable graded twisting map, then for $(X,Y)\in\{ (B_1', B_1'), (B_1', B_1''), (B_1'', B_1'')\}$ we have
$$(R\tsr 1)(1\tsr R)(X\tsr Y\tsr A_1) \subseteq A_1\tsr X\tsr Y.$$
Furthermore, if   
$$\pi_{A_1\tsr B_2\tsr B_0}(\t\tsr 1)(1\tsr \t_A)(J_2\tsr A_1)=0,$$
then
$$(R\tsr 1)(1\tsr R)(J_2\tsr A_1)\subseteq A_1\tsr J_2.$$
\end{lemma}

In applications, and in the proof, it can be helpful to note that the additional hypothesis is satisfied if and only if  $$\pi_{A_1\tsr B_2\tsr B_0}(\t\tsr 1)(1\tsr \t_A)(r(J_2)\tsr A_1)=0,$$ where $r:B_1\tsr B_1\to B_1''\tsr B_1'$ is the canonical projection. 

\begin{proof}
Note that for $a\in A_2$, $u\in B_1'$, $\t(u\tsr a)\in A_{\ge 2}\tsr B_{\le 1}$. With this in mind,
the following are straightforward to check, as in the proof of the preceding lemma. For all $u, u_1, u_2\in B_1'$, all $d, d_1, d_2\in B_1''$, and all $x\in A_1$, 
\begin{align*}
(R\tsr 1)(1\tsr R)(u_1\tsr u_2\tsr x) &= \pi_{A_1\tsr B\tsr B}(\t\tsr 1)(1\tsr \t)(u_1\tsr u_2\tsr x)\in A_1\tsr B_1'\tsr B_1'\\
(R\tsr 1)(1\tsr R)(d_1\tsr d_2\tsr x) &= \pi_{A_1\tsr B\tsr B}(\t\tsr 1)(1\tsr \t)(d_1\tsr d_2\tsr x)\in A_1\tsr B_1''\tsr B_1''\\
(R\tsr 1)(1\tsr R)(u\tsr d\tsr x) &= \pi_{A_1\tsr B\tsr B}(\t\tsr 1)(1\tsr \t)(u\tsr d\tsr x)\in A_1\tsr B_1'\tsr B_1''.
\end{align*}
These calculations prove the first part of the lemma.
It is also straightforward to check that
$$(R\tsr 1)(1\tsr R)(d\tsr u\tsr x) = \pi_{A_1\tsr B\tsr B}(\t\tsr 1)(1\tsr \t)(d\tsr u\tsr x)\in A_1\tsr B_1''\tsr B_1'$$
if and only if 
$$\pi_{A_1\tsr B_2\tsr B_0}(\t\tsr 1)(1\tsr \t_A)(d\tsr u\tsr x)=0.\hspace{.5in} (\ast)$$
Now, assume that 
$$\pi_{A_1\tsr B_2\tsr B_0}(\t\tsr 1)(1\tsr \t_A)(J_2\tsr A_1)=0.$$
Then $(\ast)$ holds for all $\rho\tsr x\in J_2\tsr A_1$.
Since $(\t\tsr 1)(1\tsr \t)(J_2\tsr A_1)\subset A\tsr J_2,$ it follows that
$$(R\tsr 1)(1\tsr R)(J_2\tsr A_1)= \pi_{A_1\tsr B\tsr B}(\t\tsr 1)(1\tsr \t)(J_2\tsr A_1)\subseteq A_1\tsr J_2.$$ 

\end{proof}

\begin{thm}
\label{separableKoszul}
Let $A$ and $B$ be quadratic algebras. Let $\t:B\tsr A\to A\tsr B$ be a separable graded twisting map. Assume $\t$ satisfies
\begin{enumerate}
\item $\pi_{A_0\tsr A_2\tsr B_1}(1\tsr \t)(\t_B\tsr 1)(B_1\tsr I_2)=0$, and 
\item  $\pi_{A_1\tsr B_2\tsr B_0}(\t\tsr 1)(1\tsr \t_A)(J_2\tsr A_1)=0.$
\end{enumerate}
Let $F$ denote the filtration on $A\tsr_{\t} B$ defined above, and let $F^B$ denote its restriction to the subalgebra $B$. Let $\widetilde{B} = {\rm gr}^{F^B}(B)$. If $A$ is Koszul, the quadratic part of $\widetilde{B}$  is Koszul and $\widetilde{B}$ has no defining relations in degree 3, then $A\tsr_{\t} B$ is Koszul.
\end{thm}

\begin{proof} 
Assume that the quadratic part of $\widetilde{B}$ is Koszul and $\widetilde{B}$ has no defining relations in degree 3. Then Theorem \ref{filtrationTrick} implies that $\widetilde{B}$ is Koszul (and $B$ is Koszul).

Let $R = \pi_{A_1\tsr B_1}\circ \t|_{B_1\tsr A_1}$. First we show that $R$ extends uniquely to a graded twisting map $\widetilde{R}:\widetilde{B} \tsr A\to A\tsr \widetilde{B}$. 
Since $\widetilde{B}_1=B_1$, we have
$$(1\tsr R)(R\tsr 1)(\widetilde{B}_1\tsr I_2)\subseteq I_2\tsr \widetilde{B}_1$$ by Lemma \ref{firstone}.
We define 
\begin{align*}
J_{(0,0,2)} &= J_2 \cap (F_{(0,0,1)}\tsr F_{(0,0,1)}) = J_2 \cap (B_1''\tsr B_1'')\\
J_{(1,0,1)} &= J_2 \cap (F_{(0,0,1)}\tsr F_{(1,0,0)} \oplus F_{(1,0,0)}\tsr F_{(0,0,1)}) \\
J_{(2,0,0)} &= J_2 \cap (F_{(1,0,0)}\tsr F_{(1,0,0)}) = J_2.
\end{align*}

The relation space of $\widetilde{B}$ is $$J_{(2,0,0)}/J_{(1,0,1)}\oplus J_{(1,0,1)}/J_{(0,0,2)}\oplus J_{(0,0,2)}.$$ 
Since $J_{(0,0,2)}\subset B_1''\tsr B_1''$, Lemma \ref{secondone} shows 
\begin{align*}
(R\tsr 1)(1\tsr R)(J_{(0,0,2)}\tsr A_1)&\subseteq A_1\tsr J_{(0,0,2)},\\
(R\tsr 1)(1\tsr R)(J_{(1,0,1)}\tsr A_1)&\subseteq A_1\tsr J_{(1,0,1)},\\
(R\tsr 1)(1\tsr R)(J_{(2,0,0)}\tsr A_1)&\subseteq A_1\tsr J_{(2,0,0)},
\end{align*}
from which it follows that
$$(R\tsr 1)(1\tsr R)((J_{(1,0,1)}/J_{(0,0,2)})\tsr A_1)\subseteq A_1\tsr (J_{(1,0,1)}/J_{(0,0,2)}),$$
and
$$(R\tsr 1)(1\tsr R)((J_{(2,0,0)}/J_{(1,0,1)})\tsr A_1)\subseteq A_1\tsr (J_{(2,0,0)}/J_{(1,0,1)}).$$
By Lemma \ref{extendingR}, $R$ extends uniquely to $\widetilde{R}$ as desired.

Next, we show that $C=A\tsr_{\widetilde{R}} \widetilde{B}$ is isomorphic to $Q=q({\rm gr}^F(A\tsr_{\t}B))$. To see this, observe that $Q$ is quadratic and its space of relations is 
$${\rm gr}^F(I_2+J_2+I_{\t})={\rm gr}^F(I_2)+{\rm gr}^F(J_2)+{\rm gr}^F(I_{\t})$$
where
\begin{align*}
{\rm gr}^F(I_2) &= I_2,\\
{\rm gr}^F(J_2) &=J_{(2,0,0)}/J_{(1,0,1)}\oplus J_{(1,0,1)}/J_{(0,0,2)}\oplus J_{(0,0,2)},\\
{\rm gr}^F(I_{\t})&={\rm span}_{\k}\{ b\tsr a - R(b\tsr a)\ : a\in A_1, b\in B_1\}.
\end{align*}
By Proposition \ref{presentation},  $C\cong Q$. 

Finally, since the Hilbert series of $A\tsr_{\t} B$ is the same as that of $C$, and hence that of $Q$, we conclude that $Q\cong {\rm gr}^F(A\tsr_{\t} B)$. If $A$ is Koszul, then Theorem \ref{pure graded Koszul} implies that $C$ is Koszul. The result now follows from Theorem \ref{filtrationTrick}.

\end{proof}

\begin{cor}
\label{free separable Koszul} 
Let $A$ be a Koszul algebra with quadratic relation space $I_2$, let $B$ be a free algebra, and suppose $\t: B \tsr A \to A \tsr B$ is a separable graded twisting map. Assume that $\pi_{A_0\tsr A_2\tsr B_1}(1\tsr \t)(\t_B\tsr 1)(B_1\tsr I_2)=0$. Then $A \tsr_{\t} B$ is Koszul.
\end{cor}

\section{Graded twisted tensor products on two generators}
\label{abc}


Let $A=\k[x]$ and $B=\k[y]$ be free $\k$-algebras and let $\tau:B\tsr A\to A\tsr B$ be a graded twisting map. Define $\t_B: B \tsr A \to A_0 \tsr B$ by $\t_B = \pi_{A_0 \tsr B} \ \t$. Suppressing the tensors, write $\tau(yx)=ax^2+bxy+cy^2$ for $a,b,c\in \k$. 

Let $s_i^d$ denote the coefficient of $y^{d}$ when $\tau(y^ix^{d-i})$ is written in terms of monomials $x^py^q$. Then for all $1\le i\le d-1$, applying the first identity from Remark \ref{AltTwistingDef} to the factorization $y^ix^{d-i+1}=y^ix^{d-i}x$ shows that 
$$\t(y^ix^{d-i+1}) - s_i^d\t(y^dx)= (\mu_A\tsr 1_B)(1_A\tsr \t)((\t-\t_B)(y^ix^{d-i})\tsr x) \in A_+\tsr B.$$
Applying the second identity from Remark \ref{AltTwistingDef} to the factorization $y^dx=y^{d-1}yx$ shows
$$\tau(y^dx) - a\tau(y^{d-1}x^2) - (bs_{d-1}^d+c)y^{d+1} \in A_+\tsr B.$$

Thus 
\begin{align*}
&(1-as_{d-1}^d)\tau(y^dx) - (bs_{d-1}^d+c)y^{d+1} \\
&=\tau(y^dx) - a\tau(y^{d-1}x^2) - (bs_{d-1}^d+c)y^{d+1} + a\left(\t(y^{d-1}x^2) - s_{d-1}^d\t(y^dx)\right)
\end{align*}
is an element of $A_+\tsr B$.

\begin{lemma}
\label{abcLemma}
If $1-ac\neq 0$, then $1-as_{d-1}^d \ne 0$ for all $d \geq 2$.
\end{lemma}

\begin{proof}
Suppose, to the contrary, that $1-as_{d-1}^d=0$ for some $d \geq 3$. Then the calculation preceding the lemma implies that $bs_{d-1}^d+c=0$. Multiplying both sides of this equation by $a$ gives $b+ac=0$. The assumption $1-ac\neq 0$ implies $b+1\neq 0$. We claim that $s^k_{k-1} = c$ for all $k \geq 2$. 

By definition, $s_1^2=c$. Suppose $s_{k-1}^k=c$. Then we have
$$(1-ac)\tau(y^kx) - (bc+c)y^{k+1} \in A_+\tsr B$$
Since $1-ac\neq 0$, it follows that $s_k^{k+1}=c(b+1)/(1-ac)=c(b+1)/(1+b)=c$, and our claim follows by induction. We have reached a contradiction since $s^d_{d-1} = c$, so $1-ac = 0$, contrary to assumption.
\end{proof}

We are ready to prove that under generic conditions on $a$, $b$, and $c$, $A\tsr_{\t} B\cong \k\la x,y \ra/\la yx-ax^2-bxy-cy^2\ra$.

\begin{thm}
\label{abcTheorem}
If $1-ac\neq 0$, then $\t$ has the unique extension property and $A\tsr_{\t} B$ is $\t$-quadratic.
\end{thm}

\begin{proof} 
By Theorem \ref{UEPquadratic}, it suffices to prove $\t$ has the unique extension property.
Fix $d\ge 2$ and assume $\t':B\tsr A\to A\tsr B$ is graded twisting to degree $d+1$ and that $\t_i=\t_i'$ for all $i\le d$. Then for every $k\le d$, the coefficient of $y^{k}$ when $\t'(y^ix^{k-i})$ is written in terms of monomials $x^py^q$ is $s_i^k$, and for all $1\le i\le d-1$,
\begin{align*}
\t(y^ix^{d-i+1}) - s_i^d\t(y^dx) &= (\mu_A\tsr 1_B)(1_A\tsr \t)((\t-\t_B)(y^ix^{d-i})\tsr x)\\
&= (\mu_A\tsr 1_B)(1_A\tsr \t')((\t-\t_B)(y^ix^{d-i})\tsr x)\\
&= (\mu_A\tsr 1_B)(1_A\tsr \t')((\t'-\t'_B)(y^ix^{d-i})\tsr x)\\
&= \t'(y^ix^{d-i+1}) - s_i^d\t'(y^dx).\qquad (\ast)
\end{align*}
The second and third equalities follow from the assumption that $\t_d=\t'_d$, and the fourth follows from the assumption that $\t'$ is graded twisting to degree $d+1$. Similarly, 

$$\t(y^dx) - a\t(y^{d-1}x^2) - (bs_{d-1}^d+c)y^{d+1}=\t'(y^dx) - a\t'(y^{d-1}x^2) - (bs_{d-1}^d+c)y^{d+1}$$

hence 

$$(1-as_{d-1}^d)\tau(y^dx) - (bs_{d-1}^d+c)y^{d+1} =(1-as_{d-1}^d)\t'(y^dx) - (bs_{d-1}^d+c)y^{d+1} $$

and

$$(1-as_{d-1}^d)\tau(y^dx)  =(1-as_{d-1}^d)\t'(y^dx).  $$

%
%

By Lemma \ref{abcLemma}, $1-as_{d-1}^d\neq 0$, so $\t(y^dx) = \t'(y^dx)$, and hence by $(\ast)$, $\t'(y^ix^{d-i+1})=\t(y^ix^{d-i+1})$ for all $1\le i\le d$. Together with the unital conditions, this shows $\t=\t'$ on a basis for $(B\tsr A)_{d+1}$.

\end{proof}
%

Now we investigate the existence of twisting maps for which $b=0$ and $ac \ne 0$. Assuming that there exists $\lambda \in \k$ such that $\lambda^2 = ac^{-1}$ we use Proposition \ref{isomorphism type of ttp} to assume that, without loss of generality, $\t(y \tsr x) = x^2 \tsr 1_B + c 1_A \tsr y^2,$ where $c \in \k$. Certain polynomials arise which we define first.

\begin{defn}
Define a sequence of polynomials $S_n(t) \in \k[t]$ by $S_1(t) = 1$, $S_2(t) = 1$, and $S_n(t) = S_{n-1}(t) - t S_{n-2}(t)$ for all $n \geq 1$.
\end{defn}

It is straightforward to prove that $S_n(t)$ has an explicit form: $$S_n(t) = \sum_{j=0}^r (-1)^j \binom{n-1-j}{j} t^j, \text{ where } r = \begin{cases} \frac{n-2}{2} &\text{ if } n \text{ is even} \\  \frac{n-1}{2} &\text{ if } n \text{ is odd}. \end{cases}$$ Moreover, these polynomials are interestingly related to the Catalan numbers \cite{EO}. It would be remarkable to uncover any deeper connection here. 


\begin{lemma}
\label{recursion}
For all $n \geq 2$ and $i, j \geq 1$,
\begin{enumerate}
\item $S_{i+1}(t) S_{i+j}(t) = t^i S_j(t) + S_{i+j+1}(t) S_i(t)$,
\item $S_n^2(t)-t^{n-1} = S_{n+1}(t) S_{n-1}(t)$.
\end{enumerate}
\end{lemma}

The proof of (1) is a straightforward double induction using the recursive definition of $S_n(t)$ and is left to the reader. Notice that (2) is a special case of (1).

For $c \in \k$, let $S_n = S_n(c)$.

\begin{thm} Suppose that $c \in \k$ is not a root of $S_n(t)$ for all $n \geq 1$. Then the linear map $\t: B \tsr A \to A \tsr B$ defined by insisting that $\t$ satisfy the unital twisting conditions and $$\t(y^i \tsr x^j) = S_{i+j}^{-1}(S_j x^{i+j} \tsr 1_B + c^j S_i 1_A \tsr y^{i+j}) \text{ for all } i, j \geq 1$$ is a graded twisting map. Moreover, $\t$ has the unique extension property.
\end{thm}

\begin{proof}
It is straightforward to check, using Theorem \ref{simpler extension condition} and Lemma \ref{recursion} (1), that $\t$ is a graded twisting map. 

We show that $\t$ has the unique extension property. Let $n \geq 2$ and suppose that $\t': B \tsr A \to A \tsr B$ is a graded linear map that is twisting to degree $n+1$ and $\t'_{\leq n} = \t_{\leq n}$. We must show that $\t'_{n+1} = \t_{n+1}$.

Since $\t'$ is twisting in degree $n+1$ we have 

\begin{align*}
\t'(y \tsr x^n) &= S_n^{-1}(S_{n-1}x^{n+1} \tsr 1_B + c^{n-1} \t'(y^n \tsr x)) \\
\t'(y^n \tsr x) &= S_n^{-1}(\t'(y \tsr x^n) + cS_{n-1} 1_A \tsr y^{n+1}).
\end{align*}

Multiplying the first equation by $S_n^{-1}$ and adding the result to the second equation yields

$$(1-c^{n-1}S_n^{-2}) \t'(y^n \tsr x) = S_n^{-2}S_{n-1}(x^{n+1} \tsr 1_B + c S_n 1_A \tsr y^{n+1}),$$ then multiplying through by $S_n^2$ and using Lemma \ref{recursion} (2) we get: $$S_{n+1}S_{n-1} \t'(y^n \tsr x) = S_{n-1}(x^{n+1} \tsr 1_B + c S_n 1_A \tsr y^{n+1}).$$ It follows that $\t'(y^n \tsr x) = \t(y^n \tsr x)$. 

Next suppose that $i+j = n+1$ and $j \geq 2$. Then 
\begin{align*}
\t'(y^i \tsr x^j) &= (\m_A \tsr 1_B)(1_A \tsr \t')(\t' \tsr 1_A)(y^i \tsr x^{j-1} \tsr x) \\
&= (\m_A \tsr 1_B)(1_A \tsr \t')(S_n^{-1}(S_{j-1} x^n \tsr 1_B + c^{j-1} S_i 1_A \tsr y^n) \tsr x) \\
&= \t(y^i \tsr x^j) \text{ since } \t'(y^n \tsr x) = \t(y^n \tsr x).
\end{align*}

We conclude that $\t$ has the unique extension property.
\end{proof}

\begin{prop}
Let $c \in \k$ be a root of $S_n(t)$ for some $n \geq 3$. Then there is no graded twisting map $\t:B \tsr A \to A \tsr B$ such that $\t(y \tsr x) = x^2 \tsr 1_B + c 1_A \tsr y^2$.
\end{prop}

\begin{proof}
Let $n \geq 3$ be minimal such that $S_n = 0$. Suppose, to the contrary, that $\t: B \tsr A \to A \tsr B$ is a graded twisting map such that $\t(y \tsr x) = x^2 \tsr 1_B + c 1_A \tsr y^2$.

Assume, inductively, that 
\begin{align*}
\t(y \tsr x^{n-2}) &= S_{n-1}^{-1}(S_{n-2} x^{n-1} \tsr 1_B + c^{n-2} 1_A \tsr y^{n-1}) \\
\t(y^{n-2} \tsr x) &= S_{n-1}^{-1}(x^{n-1} \tsr 1_B + cS_{n-2} 1_A \tsr y^{n-1}).
\end{align*}

Then using the assumption that $\t$ is a twisting map we have
\begin{align*}
\t(y \tsr x^{n-1}) &= S_{n-1}^{-1}(S_{n-2} x^n \tsr 1_B + c^{n-2} \t(y^{n-1} \tsr x)) \\
\t(y^{n-1} \tsr x) &= S_{n-1}^{-1}(\t(y \tsr x^{n-1}) + c S_{n-2} 1_A \tsr y^{n}).
\end{align*}

Isolating the term $\t(y^{n-1} \tsr x)$ exactly as in the proof of the last theorem yields $$S_n S_{n-2} \t(y^{n-1} \tsr x) = S_{n-2}(x^n \tsr 1_B + S_{n-1} 1_A \tsr y^{n}).$$ Since $S_n = 0$ and $S_{n-2} \ne 0$ this is the desired contradiction.
\end{proof}

We conclude this section by remarking that we do not have a complete picture of necessary and sufficient conditions for the existence of twisting maps in the case $b \ne 0$, $ac \ne 0$. If $\k$ is closed under taking square roots, one may assume, using Proposition \ref{isomorphism type of ttp}, that the putative twisting map satisfies $\t(y \tsr x) = cx^2 \tsr 1_B + b x \tsr y + c 1_A \tsr y^2$, but even with this simplification the analysis seems complicated.

\section{Additional Examples}

In this section we give examples. Our first example uses many of the main theorems of this paper.

%
%
%

\begin{ex}
Let $A = \k \la x, y \ra$, $B = \k \la d, u \ra$. Use Theorem \ref{extendSepTau} to define a separable graded twisting map $\t$ by  $$\t(d \tsr x) = x \tsr d + 1 \tsr d^2, \ \t(d \tsr y) = y \tsr d + 1 \tsr d^2,$$

$$\t(u \tsr x) = x \tsr u + x^2 \tsr 1, \ \t(u \tsr y) = y \tsr u + y^2 \tsr 1.$$ Also assume that $\t$ satisfies the unital twisting conditions. 

Then Proposition \ref{freeKoszul} implies that $A \tsr_{\t} B$ is Koszul and moreover it is easy to check that this algebra has global dimension equal to 2. Another nice feature of this example is that $\t$ preserves $xy - yx$ so we get via Theorem \ref{inducedTau} an induced twisting map $\t': B \tsr A' \to A' \tsr B$, where $A' = \k[x,y]$. To see this, let $I$ denote the two-sided ideal of $A$ generated by $xy-yx$. The condition $\t(B \tsr I) \subset I \tsr B$ follows from the easily verified formulas: $$\t(d \tsr (xy-yx)) = (xy-yx) \tsr d$$ $$\t(u \tsr (xy-yx)) = [(xy-yx)(x+y)+(x+y)(xy-yx)] \tsr 1_B + (xy-yx) \tsr u.$$ Finally, Corollary \ref{free separable Koszul} ensures that the algebra $A' \tsr_{\t'} B$ is Koszul and it is easy to check that this algebra has global dimension equal to 3.

\end{ex}

\begin{ex}
\label{1detNotUEP}
Here we give an example of a one-determined, graded twisting map that does not have the unique extension property. 
Let $A = \k \la x \ra/ \la x^3 \ra$ and $B = \k \la y \ra/ \la y^3 \ra$. We claim that there is a unique graded twisting map
$\t:B \tsr A \to A \tsr B$ such that $\t(y \tsr x) = x^2 \tsr 1_B - x \tsr y + 1_A \tsr y^2$.

Consider an arbitrary graded linear map $t: B \tsr A \to A \tsr B$ that satisfies the unital twisting conditions and $t(y \tsr x) = x^2 \tsr 1_B - x \tsr y + 1_A \tsr y^2$.
There exist scalars $\l, \mu, \a, \b, \gamma \in\k$ such that
\begin{align*}
t(y^2 \tsr x) &= \l x^2 \tsr y + \m x \tsr y^2,\\ 
t(y \tsr x^2) &= \a x^2 \tsr y + \b x \tsr y^2,\\
t(y^2\tsr x^2) &= \gamma x^2 \tsr y^2.
\end{align*}
It is straightforward to check that in order for $t$ to be graded twisting to degree 3, we must have $\a=\l+1$ and $\b=\m-1$. Considering the factorization $y^2\tsr x^2=(1\tsr\mu_A)(y^2\tsr x\tsr x)$ shows that for $t$ to be twisting to degree $4$ it is necessary that $\gamma=\l+\m^2$. 

Define, for all $\l, \m \in \k$, a two-parameter family of linear maps $\t_{\l, \m}: B \tsr A \to A \tsr B$ by 
\begin{align*}
\t_{\l, \m}(y \tsr x) &= x^2 \tsr 1_B - x \tsr y + 1_A \tsr y^2,\\ 
\t_{\l, \m}(y^2 \tsr x) &= \l x^2 \tsr y + \m x \tsr y^2,\\ 
\t_{\l, \m}(y \tsr x^2) &= (\l+1) x^2 \tsr y + (\m-1)x \tsr y^2.\\
\t_{\l,\m}(y^2\tsr x^2) &= (\l+\m^2)x^2y^2.\\
\end{align*}
Also assume that $\t_{\l, \m}$ satisfies the unital twisting conditions. Then $\t_{\l, \m}$ is a graded twisting map if and only if $\l = -1$ and $\m = 0$. To see this, consider
\begin{align*}
0&=\t((1\tsr \m_A)(y\tsr x\tsr x^2))\\
&=(\mu_A\tsr 1_B)(1\tsr \t)( \t(y\tsr x) \tsr x^2)\\ 
&= -(\m -1)x^2\tsr y^2 + \t(y^2\tsr x^2)\\
&=(-\m+1)x^2y^2+ (1_A\tsr \m_B)(\t\tsr 1)(y\tsr \t(y\tsr x^2)) \\
&= (\l+1)^2x^2y^2.
\end{align*}
So $\l=-1$. Then
\begin{align*}
0&=\t((\m_B\tsr 1)(y\tsr y^2\tsr x))\\
&=(1\tsr \mu_B)(\t\tsr 1)(y\tsr \t(y^2\tsr x))\\
&=(1\tsr \mu_B)(\l\t(y\tsr x^2)\tsr y+\m\t(y\tsr x)\tsr y^2)\\
&=\m x^2y^2
\end{align*}
so $\m=0.$ We leave it to the reader to complete the verification that $\t=\t_{-1,0}$ is indeed a graded twisting map; it remains to show that $\t(\mu_B\tsr 1_A)(y\tsr y\tsr x^2) = (1_A\tsr \mu_B)(\t\tsr 1_B)(1_B\tsr \t)(y\tsr y\tsr x^2)$.

Since $\t_{-1,0}$ is the only graded twisting map $B\tsr A\to A\tsr B$ such that $y\tsr x\mapsto x^2\tsr 1-x\tsr y+1\tsr y^2$, it follows that $\t$ is one-determined. However, $\t_{-1,0}$ does not have the unique extension property. Recall that the unique extension property for a graded twisting map $\t$ fails in degree $n$ whenever another linear map $\t'$ is graded twisting to degree $n$ and satisfies $(\t')_{<n}=\t_{<n}$ and $\t'_n\neq \t_n$. \emph{Note that $\t'$ need not be a graded twisting map.}

Since all of the linear maps $\t_{\l,\m}$ are graded twisting to degree 3, all are identical on $(\t_{\l,\m})_{\le 2}$, and all differ in degree 3, we see that $\t_{-1,0}$ does not have the unique extension property to degree 3. Hence $\t_{-1,0}$ does not have the unique extension property. 
\end{ex}

\begin{ex}
\label{Ugh3}
Here is an example that shows that $A \tsr_{\t} B$ Koszul does not imply that $A$ and $B$ are quadratic. In \cite{CG} the algebra $U(\g_3)$ presented as the quotient of the free algebra $\Q\la x_{ij}\ |\ 1\le i\neq j\le 3\ra$ by the homogeneous ideal generated by 
\begin{align*}
[x_{ij},x_{ik}+x_{jk}]\qquad  & i, j, k\text{ distinct }\\ 
[x_{ik},x_{jk}]\qquad & i, j, k\text{ distinct, }\\   
\end{align*}
where $[a,b]=ab-ba$, was shown to be Koszul. Furthermore in \cite{CG}  it was proved that $U(\g_3)$ is isomorphic to $A \tsr_{\t} \Q \la x_{12}, x_{21} \ra$ where $A$ is isomorphic to the subalgebra of $U(\g_3)$ generated by $\{x_{13}, x_{23}, x_{31}, x_{32}\}$; the twisting map $\t$ in this case is one-sided so the subalgebra $A$ is normal. Finally, it was proved in \cite{CG} that the algebra $A$ is not finitely presented. 
\end{ex}

\begin{ex}
\label{2sidedTTP}

Let $A = \k[x]$, $B = \k \la d,u \ra$, and let $\t:B\tsr A\to A\tsr B$ be the separable, graded twisting map uniquely determined by $$\t(d \tsr x) = x \tsr d + 1 \tsr d^2, \ \ \  \t(u \tsr x) = x^2 \tsr 1 + x \tsr u.$$
For the sake of readability, we will suppress the tensor symbol for the rest of the example. We claim that there do not exist non-trivial one-generated, graded $\k$-algebras $C$ and $D$ and a one-sided graded twisting map $\t':D\tsr C\to C\tsr D$ such that $A\tsr_{\t}B \cong C\tsr_{\t'} D$ as graded algebras. This shows that the class of twisted tensor products arising from two-sided graded twisting maps strictly contains the class arising from one-sided twisting maps. Furthermore, since $A$ and $B$ are free algebras, $A\tsr_{\t'} B$ is Koszul by Proposition \ref{freeKoszul}. 

To the contrary, suppose such $C$, $D$, and $\t'$ exist. Then there are vector space isomorphisms
$$A_1\oplus B_1 \cong (A\tsr_{\t} B)_1 \cong (C\tsr_{\t'}D)_1\cong C_1\oplus D_1$$
and, since neither $C$ nor $D$ is the trivial algebra, either $\dim_{\k} C_1 = 1$ or $\dim_{\k} D_1=1$. A dimension count in homogeneous degree 2 shows that $C$ and $D$ must be free.

There is no loss of generality in assuming $\dim_{\k} C_1=1$. To see why, first note that $A\tsr_{\t} B$ is isomorphic to its opposite algebra via $x\mapsto x, d\mapsto -d$, and $u\mapsto -u$. Thus the same is true for $C\tsr_{\t'} D$. Let $s:C\tsr D\to D\tsr C$ be the ``trivial'' twisting map $s(c\tensor d)=d\tsr c$. It is straightforward to check that $\t''=s\t' s$ is a graded twisting map $C^{op}\tsr D^{op}\rightarrow D^{op}\tsr C^{op}$. Moreover, $s$ is an algebra isomorphism $C\tsr_{\t'} D \to (D^{op}\tsr_{\t''} C^{op})^{op}$.

Let $C=\k\la a\ra$ and $D=\k\la b,c\ra$. Let $\varphi:A\tsr_{\t} B\to C\tsr_{\t'}D$ be an isomorphism. Let $\a_1, \a_2, \a_3, \b_1, \b_2, \b_3, \gamma_1, \gamma_2, \gamma_3\in \k$ such that 

\begin{align*}
\varphi(x) &= \a_1 a+\b_1 b+\gamma_1 c\\
\varphi(u) &= \a_2 a+\b_2 b+\gamma_2 c\\
\varphi(d) &= \a_3 a+\b_3 b+\gamma_3 c.\\
\end{align*}

Since $\t'$ is one-sided, either $\t'(D_1\tsr C_1)\subseteq C\tsr D_+$ or $\t'(D_1\tsr C_1)\subseteq C_+\tsr D$. We examine the latter case, leaving the former case to the reader. 

Expanding the difference $\varphi(x^2+xu)-\varphi(u)\varphi(x)$ gives

\begin{align*}
0&=\varphi(x^2+xu)-\varphi(u)\varphi(x)\\
&=\a_1^2 a^2 + (\a_1\b_1+\a_1\b_2-\a_2\b_1)ab+(\a_1\gamma_1+\a_1\gamma_2-\a_2\gamma_1)ac\\
& + \b_1^2 b^2 + (\b_1\gamma_1+\b_1\gamma_2-\b_2\gamma_1)bc + (\b_1\gamma_1+\b_2\gamma_1-\b_1\gamma_2) + \gamma_1^2 c^2\\
& + (\a_1\b_1+\a_2\b_1-\a_1\b_2)\t'(b\tsr a) + (\a_1\gamma_1+\a_2\gamma_1-\a_1\gamma_2)\t'(c\tsr a).\\
\end{align*}

By assumption, $\t'(D_1\tsr C_1)\subseteq C_+\tsr D$, so the last two terms do not contain $b^2$ or $c^2$. It follows that $\b_1=\gamma_1=0$. The analogous calculation for $\varphi(xd+d^2)-\varphi(d)\varphi(x)$ shows that $\b_3=\gamma_3=0$, leaving us with
$$0=\varphi(xd+d^2)-\varphi(d)\varphi(x)=\a_3^2 a^2$$
so $\a_3=0$, and hence $\varphi(d) =0$, a contradiction.

\end{ex}

\noindent {\bf Acknowledgement.} We thank the anonymous referee for numerous helpful suggestions.

\bibliographystyle{plain}
\bibliography{bibliog2}

\end{document}